\documentclass[preprint]{imsart}
\RequirePackage[OT1]{fontenc}
\RequirePackage{amsthm,amsmath}
\RequirePackage[colorlinks,citecolor=blue,urlcolor=blue]{hyperref}

\startlocaldefs
\numberwithin{equation}{section}
\theoremstyle{plain}
\endlocaldefs

\usepackage{amsmath}
\usepackage{amsmath, amsthm, amssymb}
\usepackage{accents}
\usepackage{graphicx}
\usepackage{color}
\usepackage{enumerate}
\usepackage{textcomp}
\usepackage{graphicx}
\usepackage{graphicx,hyperref}
\usepackage{hyperref}
\usepackage{natbib}
\usepackage{fancyhdr}
\usepackage[top=1.4in, bottom=1.4in, left=1.35in, right=1.35in]{geometry}

\newtheorem{thm}{Theorem}[section]
\newtheorem{lem}{Lemma}[section]
\newtheorem{cor}{Corollary}[section]

\newtheorem{remark}{Remark}[section]

\renewcommand{\vec}[1]{\boldsymbol{#1}}
\numberwithin{equation}{section}

\DeclareMathOperator*{\argmax}{arg\,max}

\begin{document}

\begin{frontmatter}
\title{A New Step-down Procedure for Simultaneous Hypothesis Testing Under Dependence}
\runtitle{A New Step-down Testing Procedure Under Dependence}

\begin{aug}
\author{\fnms{Prasenjit} \snm{Ghosh}\thanksref{addr1}\ead[label=e1]{prasenjit.stat@presiuniv.ac.in}}
\and
\author{\fnms{Arijit} \snm{Chakrabarti}\thanksref{addr2}\ead[label=e2]{arc@isical.ac.in}}

\runauthor{P. Ghosh and A. Chakrabarti}

\address[addr1]{Department of Statistics, Presidency University, Kolkata - 700073, India.\\ 
    \printead{e1} 
}

\address[addr2]{Applied Statistics Unit, Indian Statistical Institute, Kolkata - 700108, India.\\
    \printead{e2}
}

\end{aug}

\begin{abstract}
In this article, we consider the problem of simultaneous testing of hypotheses when the individual test statistics are not necessarily independent. Specifically, we consider the problem of simultaneous testing of point null hypotheses against two-sided alternatives for the mean parameters of normally distributed random variables. We assume that conditionally given the vector of means, these random variables jointly follow a multivariate normal distribution with a known but arbitrary covariance matrix. We consider a Bayesian framework where each unknown mean parameter is modeled through a two-component 'spike and slab' mixture prior. This way, unconditionally the test statistics jointly have a mixture of multivariate normal distributions. A new testing procedure is developed that uses the dependence among the test statistics and works in a 'step-down' manner. This procedure is general enough to be applied for non-normal data. A decision theoretic justification in favor of the proposed testing procedure has been provided by showing that unlike many traditional $p$-value based stepwise procedures, this new method possesses a certain 'convexity property' which makes it admissible with respect to a vector risk function that captures the risks for the individual testing problems. An alternative representation of the proposed test statistics has also been established resulting in great simplification in the computational complexity. It is demonstrated through extensive simulations that for various forms of dependence and a wide range of sparsity levels, the proposed testing procedure compares quite favorably with several existing multiple testing procedures available in the literature in terms of overall misclassification probability.
\end{abstract}

\tableofcontents
\end{frontmatter}

\section{Introduction}
Over the past two decades and a half, there has been a growing need to do statistical inference on large datasets involving a large number of parameters used in modeling those. Such datasets arise from modern genomic and astronomical experiments and various other fields like brain imaging, medicine, economics, finance, meteorology etc. Individual inference on multiple problems lead to errors which can accumulate to make the overall inference erroneous unacceptably often. This necessitates adjustments due to multiplicity to be made to the individual procedures to make the overall error of inference under bounds. Multiple hypothesis testing and the corresponding multiplicity adjustments have been extensively used statistical tools in recent times in the analysis of large datasets. Thus multiple testing has been an area of very active research during this period. Many different multiple testing procedures have been proposed in the literature so far, mostly with the aim of controlling some overall measure of type I error at a predetermined level $\alpha \in (0,1)$, the most notable ones being the family-wise error rate (FWER) and the False discovery rate (FDR). The FWER, defined as the probability of making at least one false rejection, is most popularly controlled by the well-known Bonferroni correction. The FWER criterion turns out to be too stringent when the number of hypotheses being tested is very large. A more pragmatic approach would be to try to control the rate of erroneous rejections instead, which is measured by the FDR and was introduced in \cite{BH1995}. FDR is the expected proportion of erroneously rejected null hypotheses among all rejections occurred, the proportion being set to zero in case of no rejection. Assuming that the test statistics are independent, \cite{BH1995} demonstrated that a specific step-up multiple testing procedure controls the FDR at the desired tolerance level $\alpha\in(0,1)$. This procedure will henceforth be referred to as the BH method. Following this seminal work, a number of other FDR controlling procedures have been proposed in the literature. Some pertinent references are \cite{BKY2006}, \cite{BL1999}, \cite{BY2001}, \cite{BLROQ2009}, \cite{GBS2009}, \cite{SKS_2002}, \cite{SKS_2007}, \cite{SKS2008}, \cite{SG2009}, \cite{STO_2002} and \cite{STS_2004}, among others.\newline

Besides their proven ability to control some overall type I error measure, some such procedures have also been shown to have some appealing theoretical properties when the test statistics are independent. As for example, see \cite{BCFG2011}, \cite{CHI2008}, \cite{FDR2009}, \cite{GW2002}, \cite{GR2008}, \cite{LRS2005} and \cite{NR2012}, in this context. But, in practice, test statistics are often correlated. It is therefore natural to ask whether such nice theoretical properties continue to hold if the assumption of independence is violated and what are the effects of dependence on the performances of such procedures. In a prominent line of research, some of these procedures and variations of them have been shown to control some overall type I error measure under certain specific forms of dependence such as ``positive regression dependence on subsets'', ``conditional dependence'' or ``weak dependence''. See, for instance, \cite{BY2001}, \cite{BLROQ2009}, \cite{GUO_2009}, \cite{LR2005}, \cite{RS2006}, \cite{SKS_2002}, \cite{SKS_2007}, \cite{SKS2008}, \cite{SG2009} and \cite{STS_2004}, among others. However, these results do not ensure the same error controlling property in other forms of dependence. In situations, when no specific assumption is made about the joint distribution of the test statistics, \cite{BY2001} proposed a certain modification of the BH method which is commonly referred to as the Benjamini-Yekutieli (BY) procedure. The aforesaid method was shown to have the FDR controlling property at a given tolerance level $\alpha$ under any arbitrary form of dependence, but this comes at the expense of the overly conservative nature of this procedure, sometimes even worse than the classical Bonferroni procedure. It should be noted that typical stepwise or fixed threshold approaches such as those cited above, are based on $p$-values derived from the marginal distributions of the individual test statistics, and hence, do not take into account the correlation between them. This may have an adverse effect on the performances of these procedures, particularly, if the correlation is {\it not} weak. For example, \cite{GGQY2007}, \cite{KY2007}, \cite{QBKY2005}, \cite{QKY2005} and \cite{QXGY2007} showed that, in presence of strong or long range dependence as in microarray experiments, the number of false discoveries and the number of false non-discoveries of standard multiple testing procedures have unacceptably high variability and hence, they become very unstable. This is similar to the observations of \cite{FH2001}, \cite{FH2002} who pointed out that when test statistics are correlated, the expected number of false discoveries of typical FWER or FDR controlling procedures need not be finite as the number of tests grows to infinity. As a result, such testing procedures may become very unreliable under dependence. See, also \cite{OWEN2005} in this context. \cite{Efron_2007} showed that failing to incorporate the dependence may lead to inferences which are quite misleading especially in high correlation structures. To sum up, traditional FDR or FWER controlling multiple testing procedures which disregard the correlation between the test statistics may lead to too many erroneous decisions and thus loss of efficiency. On the other hand, based on information theoretic arguments, \cite{HJ2010} opined that the presence of correlation can indeed be a ``blessing'' rather than a ``curse'' (as perceived by many) since it provides more information about the uncertainty of the random observable as compared to the case of independence. Therefore, through a careful exploitation of such information, one may build better multiple testing procedures in terms of enhanced power. Similar observations have been made by \cite{BH2007}, \cite{BH2000} and \cite{GRW2006}. The foregoing discussion highlights the importance of taking into account the effect of correlation in deriving a multiple testing procedure when the test statistics for the individual problems are dependent. Several works have appeared in the literature where the dependence is 
factored in while deriving a multiple testing procedure. Notable references include \cite{Efron_2007}, \cite{FHG_2012}, \cite{FKC_2009}, \cite{GHS2014}, \cite{LEEK_STO_2008}, \cite{POLVDL2002},  \cite{ROM_SHK_WOLF_2008}, and \cite{YB1999}, among others. However, a final answer to this problem is yet to come and the issue remains an important and challenging open problem for researchers in this domain. One of the main goals of this paper is to suggest a new procedure which directly takes into account the dependence among individual test statistics and works under arbitrary forms of dependence.\newline

It would be worth noting, as pointed out in \cite{COHEN_SACK_2005_A} and \cite{SUN_CAI_2009}, that the major emphasis of research in multiple testing under dependence has been on finding appropriate testing procedures which could control some overall type I error rate, while questions of the decision theoretic validity (for example, admissibility) and optimality of such procedures have not been addressed adequately. \cite{COHEN_SACK_2005_A}, \cite{DSB2003} and  \cite{FS2002} strongly argued that investigating such properties is not only essential for comparing their performances but is also important to obtain deeper insights about their behaviors in such compound decision problems. A case in point is the performance of the celebrated BH method which was shown to possess several optimality properties assuming independence of test statistics. See, for instance, \cite{BCFG2011}, \cite{CHI2008}, \cite{FDR2009}, \cite{GW2002}, \cite{GR2008}, \cite{LRS2005} and \cite{NR2012} in this context. In contrast to that, \cite{COHEN_KOL_SACK_2007}, \cite{COHEN_SACK_2005_B}, \cite{COHEN_SACK_2007} and \cite{COHEN_SACK_2008} showed that, in many commonly occurring situations when test statistics are dependent, typical $p$-value based stepwise testing procedures including the BH method, are inadmissible for testing against one and two sided alternatives across a variety of loss functions. However, many of these procedures continue to have some type I error controlling property under such set up. Motivated by this we also want to study in this paper the decision theoretic aspect of multiple testing under dependence.\newline

We want to study the problem of multiple hypothesis testing under dependence in the following context. Our modeling will be a Bayesian one and our theoretical investigations will be rooted in decision theory. Suppose we have a vector $\vec{X}=(X_1,\dots,X_n)$ of observations such that given the vector $\vec{\theta}=(\theta_{1},\dots,\theta_{n})$ of unknown means, $\vec{X}$ has a multivariate normal $N(\vec{\theta},\vec{\Sigma})$ distribution, where $\vec{\Sigma}$ is an $n\times n$ positive definite matrix which we assume to be known but is arbitrary. We are interested in testing $H_{0i}:\theta_i=0\mbox{ against } H_{1i}:\theta_i\neq 0$ simultaneously for $i=1,\dots,n$. We model the unknown $\theta_i$'s is through the natural two-groups mixture prior and take as the overall loss of a multiple testing rule as the additive loss which adds up the losses corresponding to each individual test (see Section 2 for the details and motivation). This loss was first suggested by \cite{LEHMANN1957a} and \cite{LEHMANN1957b} in the multiple testing literature. The optimal Bayes rule obtained under this loss rejects a null hypothesis if its posterior probability falls below a specified threshold. However, as commented in \cite{XCML_2011}, due to the complex form of the posterior probabilities under such two-groups formulation, implementation of the optimal Bayes rule would be computationally prohibitive, even if $n$ is moderate. For the same reason, investigation of the risk properties of the Bayes rule and other multiple testing procedures becomes very difficult. The present multiple testing problem was taken up in \cite{XCML_2011} where they assumed $\vec{\Sigma}$ to be a known positive definite matrix with short range dependence. As in our case, they considered a hierarchical Bayesian framework where the individual $\theta_i$'s were modeled through a two-group point mass mixture prior and proposed a multiple testing procedure for the problem. Their proposed methodology was shown to asymptotically control the marginal false discovery rate (mFDR), defined as the proportion of the expected number of false discoveries to the expected number of discoveries. Their method performs well for the short range dependence case, but is not expected to have similar performance in other forms of dependence. Hence, the scope of application of their method is too restrictive. \cite{CHEN_SARKAR_2004} proposed a Bayesian analogue of the frequentist step-down procedures. Their proposed method consists of two steps: the first step is to rank the null hypotheses according to an increasing order of their marginal Bayes factors, and the second step is to proceed in the traditional step-down manner based on a set of conditional Bayes factors. The marginal Bayes factor of each individual null hypothesis has a one-to-one correspondence with the posterior probability of that corresponding null being true. Hence, the step-down method due to \cite{CHEN_SARKAR_2004} is also computationally very demanding in this context. Moreover, the aforesaid procedure, although intuitively very appealing, does not have any formal decision theoretic justification in favor of its use. Hence, we intend to develop a Bayesian testing procedure which is intuitive and easier to implement, which can fully incorporate the correlation between the test statistics and has desirable theoretical properties from decision theoretic viewpoints. \newline

 
We propose in this article a novel Bayesian multiple testing procedure that works in a step-down manner, henceforth referred to as the Bayesian Step Down (BSD) procedure. Our procedure has several appealing features as a multiple testing procedure. {\it Firstly}, the BSD method fully utilizes the dependence between the test statistics $\vec{X}$ at every stage and is applicable under arbitrary known covariance matrices. Most of the multiple testing procedures available in the literature do not incorporate such information. {\it Secondly}, the BSD method can be applied and has desirable performance (both theoretically and in simulations) under arbitrary form of dependence when $\vec{\Sigma}$ is known. It may be recalled that some of the well known multiple testing procedures in the literature are only meaningful for some special form of dependence, for example, positive regression dependence, among the test statistics. As a matter of fact, although the underlying testing algorithm has been developed for the dependent normal means problem, it is a generic multiple testing algorithm that can be applied to non-normal models also such as the multivariate-$t$. {\it Thirdly}, and perhaps most importantly, unlike many other well known multiple testing procedures such as the widely popular BH method, use of the BSD procedure can be justified based on decision theoretic considerations. Under our assumed setup, \cite{MAT_TRU_1967} provided a certain convexity criterion which is both necessary and sufficient for a multiple testing procedure to be admissible with respect to a vector loss function, where the component-wise losses correspond to the usual $0-1$ loss function in the standard single hypothesis testing problems. \cite{COHEN_SACK_XU_2009} proposed a step-down testing procedure from a frequentist viewpoint, which they referred to as the Maximum Residual Down (MRD) method. We observe that there exists a close connection between the BSD method and the MRD method by showing a functional relationship between the BSD statistics and the corresponding MRD statistics. Exploiting this connection we are able to deduce the fact that for the present testing problem (\ref{TESTING_PROBLEM}), our proposed BSD method possesses the desirable convexity property, and is therefore admissible in the sense described above. This provides an important decision theoretic justification in favor of our proposed method. To the best of our knowledge, full-fledged optimality study of Bayesian testing procedures (other than the Bayes rule under the additive loss) within such decision theoretic framework for multiple testing under dependence is new in the literature. We emphasize here that the aforesaid admissibility property of the BSD method does not follow as a direct consequence of that of the MRD procedure proved in \cite{COHEN_SACK_XU_2009}. As will be evident later in this paper that we need a general technique invoking some novel arguments to adapt the basic architecture of the arguments in \cite{COHEN_SACK_XU_2009} in our context. This also makes proofs of some results in \cite{COHEN_SACK_XU_2009} more explicit. As a matter of fact, it follows from our arguments that any step-down multiple testing procedure based on a set of statistics which are non-decreasing functions of the absolute values of the corresponding MRD statistics is also admissible under the aforesaid vector loss function. This is a new fact that generalizes part of the results of \cite{COHEN_SACK_XU_2009}. {\it Fourthly}, our proposed methodology is easily implementable and can avoid Markov Chain Monte Carlo type computations which can often be very demanding from a computational point of view, specially when the number of tests $n$ is very large. This is explained in more detail in the next paragraph. We also investigate the performance of our proposed method through simulations based on different choices of $\vec{\Sigma}$ which cover various strong and weak correlation structures. Our simulation results provide strong numerical evidence to show that, for every choice of $\vec{\
Sigma}$ considered in our simulation study, the Bayes misclassification risk of the proposed BSD method is considerably lower together with enhanced power as compared to several existing multiple testing procedures (including the BH method) available in the literature.\newline

We also provide an important and useful representation of the proposed BSD test statistics. Using this representation at every stage of our proposed methodology, one only needs to find the inverse of a certain sub-matrix of $\vec{\Sigma}$ and subsequent computation of all the BSD statistics for the corresponding step becomes almost immediate on modern computing platforms, even if $n$ is large. This reduces the overall computational complexity of the BSD method by a large extent and helps avoiding Markov Chain Monte Carlo type computations which can be very expensive in high dimensional problems. Such a representation works for any form of the covariance matrix $\vec{\Sigma}$. This would be particularly very useful when $\vec{\Sigma}$ corresponds to an intraclass correlation and a block (clumpy) dependence matrix. In particular, for the intraclass correlation model, we do not even require inversion of any matrix and thus the BSD method can be applied for any arbitrarily large $n$. Due to the functional relationship between the BSD statistics and the MRD statistics, the aforesaid computational savings applies equally for the MRD method. This amounts to huge computational savings for the MRD method compared to its original formulation as in \cite{COHEN_SACK_XU_2009}.\newline


The organization of this paper is as follows. Section 2 gives our prior specification and the motivation towards the development of the proposed Bayesian Step Down procedure. Section 2.1 provides the formal description of our proposed methodology. Section 3 provides various theoretical results concerning the admissibility property of the BSD method. Equivalent representation of the BSD statistics and associated results are given in Section 4. Performance of the BSD method based on simulation studies for various choices of the covariance matrix $\vec{\Sigma}$ is presented in Section 5, followed by some concluding remarks in Section 6. Proofs of all the theoretical results of this paper are presented in the Appendix.


\section{Preliminaries and the Bayesian Step Down (BSD) procedure}

As mentioned before, in the present paper, we consider the problem of simultaneous testing for means of a set of jointly normal variables. Recall that we assume observing the vector $\vec{X}=(X_1,\dots,X_n)$ (obtained through some suitable transformation, if necessary) such that $\vec{X}|\vec{\theta} \sim N_{n}(\vec{\theta},\vec{\Sigma})$ where  $\vec{\theta}=(\theta_1,\dots,\theta_n)$ is the vector of unknown means and $\vec{\Sigma}=((\sigma_{ij}))$ is an $n \times n$ known positive definite matrix with an arbitrary covariance structure. We are interested in testing simultaneously 
\begin{equation}\label{TESTING_PROBLEM}
  H_{0i}:\theta_{i}=0 \mbox{ against } H_{Ai}:\theta_{i}\neq 0, \mbox{ for } i=1,\dots,n.
 \end{equation}

Note that since $\vec{\Sigma}$ is known, without loss of generality, one may assume $\vec{\Sigma}$ to be the correlation matrix of the $X_i$'s so that $X_i \sim N(\theta_{i},1)$ for each $i=1,\dots,n$. This is so since if
 \begin{eqnarray}\label{DIAG_MATRIX}
 D
 &=& \begin{pmatrix} \sigma_{11} & 0 & 0 & \dots & 0 \\ 0 & \sigma_{22} & 0 & \dots & 0 \\ \vdots \\ 0 & 0 & 0 & \dots & \sigma_{nn}\end{pmatrix},
\end{eqnarray}
then letting $\vec{U}=D^{-1/2}\vec{X}$ we have $\vec{U}\sim N_{n}(\vec{\mu}, \Lambda)$, where $\vec{\mu} = D^{-1/2} \vec{\theta}$ and $\Lambda=D^{-1/2}\vec{\Sigma} D^{-1/2}$ is simply the correlation matrix of $\vec{X}$. Therefore, testing $H^{'}_{0i}:\mu_i=0\mbox{ against } H^{'}_{Ai}:\mu_i\neq 0$ simultaneously for $i=1,\dots,n$, is equivalent to the original testing problem (\ref{TESTING_PROBLEM}).\newline

For modeling the $\theta_i$'s, we adopt the same two-groups mixture framework as in \cite{XCML_2011}. Towards that, for each $i=1,\dots,n$, let us define an indicator variable $\nu_i$ which takes the value 1 if and only if $H_{Ai}$ is true and 0 otherwise. Here $\nu_1,\dots,\nu_n$ are unobservable. It is assumed that $\nu_i \stackrel{i. i. d.}{\sim}$ Bernoulli$(p)$ for some $p \in (0, 1)$. The parameter $p$ is interpreted as the theoretical proportion of true alternatives. Given $\nu_i=0$, $\theta_{i}$ is assumed to follow the distribution $\delta_{\{0\}}$ degenerate at the point $0$, while given $\nu_i=1$ it is assumed to have an absolutely continuous distribution with density $g(\cdot)$ over $\mathbb{R}$. Thus $\theta_{i}$'s are modeled as independent and identically distributed observations from the following two-groups prior distribution:
  \begin{equation}\label{TWO_GRP_THETA_PRIOR}
  \theta_{i} \stackrel{i. i. d.}{\sim} (1-p)\cdot\delta_{\{0\}} + p\cdot g(\theta), \mbox{ for } i=1,\dots,n.
  \end{equation}
  
The corresponding common marginal distribution of $X_i$'s has the following density:
   \begin{equation}\label{TWO_GRP_MARGINAL_Xi}
  X_i \sim (1-p)\cdot f_0(x) + p\cdot f_1(x), \mbox{ for } i=1,\dots,n,
  \end{equation} 
 where $f_0 =\phi$ and $f_1(x)=\int_{\mathbb{R}}\phi(x-\theta)g(\theta)d\theta$ is the convolution of $g(\cdot)$ with the standard normal probability density function $\phi(\cdot)$. We choose $g$ as a univariate normal density with location zero and a large variance $V$. The large variance is taken to facilitate efficient detection of the non-null $\theta_i$'s. The corresponding joint conditional distribution of $\vec{X}=(X_1,\dots,X_n)$ given $\vec{\nu}=(\nu_1,\dots,\nu_n)$ is then given by
   \begin{equation}\label{CONDITIONAL_DISTRIBUTION_X}
    \vec{X}|\vec{\nu} \sim  N_{n}(0,\vec{\Sigma}+VB_{\vec{\nu}})
  \end{equation}
  and the marginal joint distribution of $ \vec{X}=(X_1,\dots,X_n)$ is given by,
  \begin{equation}\label{MARGINAL_DISTRIBUTION_X}
    \vec{X} \sim \sum_{\vec{\nu} \in \{0,1\}^{n} }\pi(\vec{\nu}) N_{n}(0,\vec{\Sigma}+VB_{\vec{\nu}})
  \end{equation}
  where $\pi(\vec{\nu})= \prod_{i=1}^{n} p^{\nu_i}(1-p)^{1-\nu_i} $ denotes the joint prior distribution of $(\nu_1,\dots,\nu_n)$ and $B_{\vec{\nu}}$ is a diagonal matrix with diagonal elements $\nu_1,\dots,\nu_n$ respectively.\newline 

It is easy to see that our multiple testing problem (\ref{TESTING_PROBLEM}) is equivalent to test the following $n$ hypotheses simultaneously:
  \begin{equation}\label{TESTING_PROBLEM_EQUIV}
   H_{0i}:\nu_i = 0 \mbox{ against } H_{Ai}:\nu_i = 1, \mbox{ for } i=1,\dots,n.
  \end{equation}

As mentioned earlier, one of our main emphasis in this work would be a study of the problem from a decision theoretic point of view. Suppose we define the loss of a multiple testing procedure as an additive one which adds up the losses made by the induced individual testing rules. At the individual testing level, suppose we define the loss to be zero if a correct decision is made and take it to be equal to 1 if a type I or type II has been committed. The theoretical optimal solution to the above multiple testing problem (\ref{TESTING_PROBLEM_EQUIV}) would be the Bayes rule with respect to the two-groups prior (\ref{TWO_GRP_THETA_PRIOR}), given by,
\begin{equation}\label{OPTIMAL_BAYES_RULE}
\nu_i^{*}=I\{\pi(\nu_i=1|\vec{x}) > \delta\},\mbox{ for } i=1,\dots,n, 
\end{equation}
for some appropriate thresholding constant $\delta>0$, where $I\{A\}$ denotes the indicator function of an event $A$. In (\ref{OPTIMAL_BAYES_RULE}) above, $\pi(\nu_i=1|\vec{x})$ denotes the posterior probability of the $i$-th alternative hypothesis being true, where by definition
 \begin{eqnarray}\label{POST_INCL_PROB}
   \pi(\nu_i=1|\vec{x})&=& \frac{\sum_{\vec{\nu}\in \{0,1\}^{n}:\nu_i=1} \pi(\vec{\nu})f(\vec{x}|\vec{\nu})}{\sum_{\vec{\nu}\in\{0,1\}^{n}}\pi(\vec{\nu})f(\vec{x}|\vec{\nu})}\cdot
 \end{eqnarray}

In (\ref{POST_INCL_PROB}) above, $f(\vec{x}|\vec{\nu})$ denotes the conditional density of $\vec{X}$ given $\vec{\nu}=(\nu_1,\dots,\nu_n)$ evaluated at the point $\vec{x}$. As commented in \cite{XCML_2011}, in order to compute $\pi(\nu_i=1|\vec{X})$ one needs to sum over $2^{n}$ many terms and this would be true for each $i=1,\dots,n$. As a result, the overall computational complexity of the optimal Bayes solution (\ref{OPTIMAL_BAYES_RULE}) will be of the order of $O(n2^{n})$, which would be computationally very expensive even if the conditional density $f(\vec{x}|\vec{\nu})$ is completely specified and the number of hypotheses $n$ is moderately large. It should be noted further in this context that, because of the same reason, derivation of the closed form expression of the optimal Bayes risk (even asymptotically) and investigation of the risk properties of other multiple testing procedures (both theoretically and numerically) as compared to this optimal risk, become practically impossible under a general covariance structure. So studying the exact optimal rule corresponding to the additive loss function is beyond our scope in this scenario. But as we will see later, this loss function (giving the overall misclassification rate) will be used in our simulations for evaluating competing procedures. \newline

As mentioned earlier, to get around the difficulty, \cite{XCML_2011} proposed a step-up method as an alternative Bayesian testing procedure where they assumed $\vec{\Sigma}$ to be a known positive definite matrix with a short range dependence covariance structure and derived some optimal properties of their procedure in terms of asymptotic control of mFDR. But the scope of application of their method is limited by the fact that it is not expected to perform well under general dependence. In a different context, when observations are independent, \cite{CHEN_SARKAR_2004} proposed a novel Bayesian step-down testing procedure based on a set of conditional Bayes factors. According to their proposal, one first considers a set of marginal Bayes factors $B_{1},\dots,B_{n}$, where each $B_{i}$ is defined as the ratio of marginal posterior odds to the marginal prior odds of $H_{0i}$. Note that, for each $i$, $B_{i}$ provides a measure of evidence in favor of $H_{0i}$ that is contained in the data $\vec{x}$. Based on the increasing ordering of these marginal Bayes factors $B_{i}$'s, \cite{CHEN_SARKAR_2004} then considered a family of $(n+1)$ most plausible configurations of true and false null hypotheses. For each such configuration, they defined a conditional Bayes factor which acts as an {\it a posteriori} measure of evidence for that particular configuration compared to the rest of the members in the aforesaid family. \cite{CHEN_SARKAR_2004} referred to these conditional Bayes factors as the stepwise Bayes factors and proposed a Bayesian testing algorithm that works in a step-down fashion by comparing the stepwise Bayes factors with a predetermined threshold $c>0$. They considered such a step-down testing procedure to be a natural Bayesian analogue to the frequentist step-down procedures. It should be noted that, although their proposed step-down testing procedure was originally developed and implemented under the assumption of independence, it is also philosophically applicable in more general contexts, such as the present multiple testing problem (\ref{TESTING_PROBLEM_EQUIV}). Their method, however, critically hinges upon enumeration of the marginal Bayes factors $B_{i}$'s and the subsequent ordering of the null hypotheses based on them. Observe that, in our context, $B_{i}=\frac{p}{1-p}\cdot\frac{\pi(\nu_{i}=0|\vec{x})}{1-\pi(\nu_{i}=0|\vec{x})}$ for $i=1,\dots,n$. Here, each $B_{i}$ has an one-to-one correspondence to the posterior probability $\pi(\nu_i=1|\vec{x})=1-\pi(\nu_i=0|\vec{x})$. Thus, implementation of the Bayesian step-down procedure due to \cite{CHEN_SARKAR_2004} faces the same computational difficulty as the optimal Bayes rule (\ref{OPTIMAL_BAYES_RULE}) for the present multiple testing problem.\newline

Given this background, we now motivate the development our proposed step-down testing procedure, henceforth referred to as the Bayesian Step Down (BSD) procedure. The formal description of this procedure is given in Section 3.1. We emphasize that we now have a moderate goal of coming up with a procedure which performs well under arbitrary dependence vis-a-vis the procedures available in the literature with respect to natural losses and also has attractive decision theoretic properties (to be explained in the next section). Our route will be Bayesian, although we will not be using formal optimal Bayes rules for reasons explained before. It is worth noting that in the frequentist literature, a step-down procedure starts by determining whether the null hypothesis corresponding to the most significant test statistic can be rejected or not. Thus, in a step-down method, we try to answer the following question at the first step: ``Can at least one null hypothesis be rejected?'' which is equivalent to asking the question ``Can the global null hypothesis be true?''. A natural Bayesian approach to answer this question is as follows. To compare the global null hypothesis, we confine our attention to the sub-space $\{(\nu_1,\dots,\nu_n)\in \{0,1\}^{n}:\sum_{i=1}^{n}\nu_i=1\}$ of the original model space $\{0,1\}^{n}$, as the class of plausible alternatives to the global null. That is, we are considering only those models as plausible alternatives to the global null hypothesis, each of which consists of $(n-1)$ many true null and exactly one false null hypothesis. For each of these models in this restricted sub-space, we enumerate the ratio of the posterior probability of an alternative model being true to that of the global null hypothesis. If the maximum of these ratios exceeds some pre-specified threshold, say, $\delta > 0$, we conclude that there is little reason to believe the global null hypothesis to be true in light of the data. Suppose the maximum occurs for the odds ratio corresponding to the alternative $(\nu_1=0,\dots,\nu_{i-1}=0,\nu_i=1,\nu_{i+1}=0,\dots,\nu_n=0)$, for $1 \leq i \leq n$. Then we reject $H_{0i}$ and leave aside the corresponding $x_i$ for further analysis, and move on to the next stage. Otherwise, we accept all the null hypotheses and hence the global null, and stop. We continue in this fashion till an acceptance occurs or we exhaust considering all the null hypotheses.\newline


We now formally describe the proposed Bayesian step-down procedure as follows. For that, we adopt here similar convention of notations used in \cite{COHEN_SACK_XU_2009}. Let $\vec{X}^{(i_{1},\dots,i_{t})}$ be an $(n-t)\times1$ vector consisting of those components of $\vec{X}=(X_1,\dots,X_n)$ with $(X_{i_{1}},\dots,X_{i_{t}})$ left out. Suppose $\vec{\Sigma}_{(i_{1},\dots,i_{t})} $ is the $(n-t)\times(n-t)$ sub-matrix obtained after eliminating the $i_{1},\dots,i_{t}$-th rows and the corresponding columns of $\vec{\Sigma}$. Let $\vec{\sigma}_{(j)}^{(i_{1},\dots,i_{t})} $ be the $(n-t-1)\times1$ vector obtained by eliminating the $i_{1},\dots,i_{t}$-th and $j$-th elements of the $j$-th column vector of $\vec{\Sigma}$. Further suppose that
 \begin{eqnarray}
 \sigma_{j\cdot(i_{1},\dots,i_{t})} &=& \sigma_{jj} - {\vec{\sigma}_{(j)}^{(i_{1},\dots,i_{t})}}^{T}{\vec{\Sigma}^{-1}_{(i_1,\dots,i_t,j)}}{\vec{\sigma}_{(j)}^{(i_{1},\dots,i_{t})}}. \nonumber
 \end{eqnarray}
 
 Let us define
 \begin{eqnarray}\label{BSD_STAT_DEFN}
 S_{tj}^{(i_1,\dots,i_{t-1})}(\vec{X}) &=& \frac{\pi(\nu_j=1,\vec{\nu}^{(i_1,\dots,i_{t-1}, j)}=\vec{0}|\vec{X}^{(i_1,\dots,i_{t-1})})}{\pi(\nu_j=0,\vec{\nu}^{(i_1,\dots,i_{t-1},j)}=\vec{0}|\vec{X}^{(i_1,\dots,i_{t-1})})}
 \end{eqnarray}
for $t,j=1,\dots,n,$ $1\leqslant i_{1}\neq\dots\neq i_{t-1}\leqslant n$ and $i_{l} \neq j$ for all $l=1,\dots,t-1$.\newline

Note that, for each fixed $t$, the numerator of the statistics $S_{tj}^{(i_1,\dots,i_{t-1})}(\vec{X})$ is nothing but the posterior probability of the $j$-th plausible alternative within the restricted subspace $\{\vec{\nu}^{(i_{1},\dots,i_{t-1})}\in \{0,1\}^{n-t+1}:\sum_{i}\nu_i=1\}$, while the denominator is the posterior probability of the corresponding global null hypothesis. Thus, for each $t$, the statistics $S_{tj}^{(i_1,\dots,i_{t-1})}(\vec{X})$'s are simply the ratios of posterior probabilities of the plausible alternatives within the aforesaid restricted subspace to that of the corresponding global null hypothesis at the $t$-th stage.\newline
 
For $t=1,\dots,n$, let us define the indices $j_t(\vec{X})$ as,
 \begin{align}\label{BSD_INDEX}
  j_t(\vec{X})&=\argmax_{j \in \{1,\dots,n\}\setminus\{j_1(\vec{X}),\dots,j_{t-1}(\vec{X})\}} S^{(j_1(\vec{X}),\dots,j_{t-1}(\vec{X}))}_{tj}(\vec{X}).
 \end{align}

 \subsection{The Proposed Procedure}
 
 Given a predetermined threshold $\delta > 0$, the proposed Bayesian Step Down (BSD) procedure is now described below.\newline
 
 \begin{enumerate}
 \item At stage 1, consider the statistics $S_{1j}(\vec{X})$, where $j\in\{1,\dots,n\}$. If $S_{1j_1(\vec{X})}(\vec{X})\leqslant \delta$, stop and accept all the $H_{0i}$'s. Otherwise, reject $H_{0j_1(\vec{X})}$ and continue to stage 2.\newline
 
 \item At stage 2, consider the statistics $S^{(j_1(\vec{X}))}_{2j}(\vec{X})$, where $j\in \{1,\dots,n\}\setminus\{j_1(\vec{X}) \}$. If $S^{(j_1(\vec{X}))}_{2j_2(\vec{X})}(\vec{X})\leqslant\delta$, stop and accept all the remaining $H_{0i}$'s. Otherwise, reject $H_{0j_2(\vec{X})}$ and continue to stage 3.\newline
 
 \item In general, at stage $t$, consider the $(n-t+1)$ many statistics $S^{(j_1(\vec{X}),\dots,j_{t-1}(\vec{X}))}_{tj}(\vec{X})$, where $j \in \{1,\dots,n\}\setminus\{j_1(\vec{X}),\dots,j_{t-1}(\vec{X})\}$. If $S^{(j_1(\vec{X}),\dots,j_{t-1}(\vec{X}))}_{tj_t(\vec{X})}(\vec{X})\leqslant\delta$, stop and accept all the remaining $H_{0i}$'s. Otherwise, reject $H_{0j_t(\vec{X})}$ and move to stage $(t+1)$.\newline
 
  \item We continue in this way till an acceptance occurs or we are exhausted with all the null hypotheses (that is $t=n$), in which case we must stop.
\end{enumerate}
 
 Here the subscript $t$ denotes the stage of the BSD procedure. The above description defines a class of Bayesian testing procedures for various choices of the thresholding constant $\delta>0$. It should be noted that, at each step $t$, the statistics $S_{tj}^{(i_1,\dots,i_{t-1})}(\vec{X})$'s directly incorporate the dependence among the $X_i$'s by using either the covariance matrix $\vec{\Sigma}$ or its various sub-matrices of appropriate orders which appear in the posterior probabilities through the corresponding likelihoods. It should further be noted that, although we remove the data points $X_{i_1},\dots,X_{i_{t-1}}$ before the $t$-th stage is reached, each of the statistics $S_{tj}^{(i_1,\dots,i_{t-1})}(\vec{X})$'s at the $t$-th stage implicitly depends on the observations $X_{i_1},\dots,X_{i_{t-1}}$ through the indices $j_1(\vec{X}),\dots,j_{t-1}(\vec{X})$ for which the corresponding null hypotheses have already been rejected up to the $(t-1)$-th stage. Thus, the decision taken at each and every step has an effect on the decisions taken in the following steps, and all the data points are being used simultaneously at every step. In this fashion, we are incorporating the dependence among the test statistics in a more fruitful manner compared to the traditional multiple testing procedures available in the existing literature. \newline
    
 Observe that for $d\in \{0,1\}$,
 \begin{eqnarray}
  \pi\big(\nu_j=d,\vec{\nu}^{(i_1,\dots,i_{t-1}, j)}=\vec{0}|\vec{X}^{(i_1,\dots,i_{t-1})}\big)
  &=&\pi\big(\nu_j=d|\vec{\nu}^{(i_1,\dots,i_{t-1},j)}=\vec{0},\vec{X}^{(i_1,\dots,i_{t-1})}\big)\nonumber\\
  && \mbox{ }\times \mbox{ } \pi\big(\vec{\nu}^{(i_1,\dots,i_{t-1}, j)}=\vec{0}|\vec{X}^{(i_1,\dots,i_{t-1})}\big).\nonumber
 \end{eqnarray}

 Thus, one may rewrite the statistics $S_{tj}^{(i_1,\dots,i_{t-1})}(\vec{X})$ as
 \begin{eqnarray}
 S_{tj}^{(i_1,\dots,i_{t-1})}(\vec{X}) &=& \frac{\pi(\nu_j=1|\vec{\nu}^{(i_1,\dots,i_{t-1}, j)}=\vec{0},\vec{X}^{(i_1,\dots,i_{t-1})})}{\pi(\nu_j=0|\vec{\nu}^{(i_1,\dots,i_{t-1}, j)}=\vec{0},\vec{X}^{(i_1,\dots,i_{t-1})})}\nonumber
 \end{eqnarray}
 for $t,j=1,\dots,n,$ $ 1 \leqslant i_{1} \neq \dots \neq i_{t-1} \leqslant n$ and $i_{l} \neq j $ for all $l=1,\dots,t-1$. Hence, the BSD test statistics at stage $t$ may also be interpreted as the ratios of posterior probabilities of $H_{Aj}$ being true and $H_{0j}$ being true, provided that the rest of the null hypotheses which have not been rejected before the $t$-th stage are true. 
   
\section{Admissibility property of the BSD Procedure}
 In this section, we provide an important decision theoretic justification in favor of the use of the BSD procedure from a frequentist viewpoint in our setup.  In particular, we show that the proposed method based on the statistics $S_{tj}$'s in (\ref{BSD_STAT_DEFN}), will be admissible in a sense to be made precise shortly. Recall here that, any multiple testing procedure $\Phi(\vec{x})=(\phi_1(\vec{x}),\dots,\phi_n(\vec{x}))$ induces an individual test function $\phi_j(\vec{x})$ for testing $H_{0j}$ against $H_{Aj}$, where $\phi_j(\vec{x})$ denotes the probability of rejecting the $j$-th null hypothesis $H_{0j}$ when the data point $\vec{X}=\vec{x}$ is observed. We consider the standard $0-1$ loss function corresponding to $\phi_j$ which is given by
 \begin{equation}\label{INDIVIDUAL_LOSS}
  L_j\big(\phi_j(\vec{X}),\vec{\theta}\big)=I\{\theta_j=0\}\phi_j(\vec{X})+I\{\theta_j\neq0\}(1-\phi_j(\vec{X})\big),
 \end{equation}
while the corresponding risk function will be given as
 \begin{equation}\label{INDIVIDUAL_RISK}
  R_j\big(\phi_j,\vec{\theta}\big)=I\{\theta_j=0\}E_{\vec{\theta}:\theta_j=0}\big(\phi_j(\vec{X})\big) + I\{\theta_j\neq0\}E_{\vec{\theta}:\theta_j\neq0} \big(1-\phi_j(\vec{X})\big).\nonumber
 \end{equation}

 We consider the loss function for the procedure $\Phi(\vec{X})$ to be defined as the vector loss function
 \begin{equation}\label{VECTOR_LOSS}
  L\big(\Phi(\vec{X}),\vec{\theta}\big)=(L_1\big(\phi_1(\vec{X}),\vec{\theta}\big),\dots,L_n\big(\phi_n(\vec{X}),\vec{\theta}\big)),
 \end{equation}
 while the corresponding risk function is defined as the vector of the individual risk functions and is given by
 \begin{equation}\label{VECTOR_RISK}
  R\big(\Phi,\vec{\theta}\big)=(R_1\big(\phi_1,\vec{\theta}\big),\dots,R_n\big(\phi_n,\vec{\theta}\big)).\nonumber
 \end{equation}

 A multiple testing procedure $\Phi(\vec{X})$ is said to be inadmissible with respect to the vector loss function (\ref{VECTOR_LOSS}) if there exists another multiple testing procedure $\Phi^{*}(\vec{X})$ such that $R_j\big(\phi^{*}_j,\vec{\theta}\big) \leq R_j\big(\phi_j,\vec{\theta}\big)$ for all $j=1,\dots,n$, and all $\vec{\theta} \in \mathbf{R}^n$, with strict inequality holding for at least one $j$ and some $\vec{\theta}\in \mathbf{R}^n$. A multiple testing procedure will be admissible if it is not inadmissible in the aforesaid sense. It is natural that a multiple testing procedure which is inadmissible with respect to the vector loss function (\ref{VECTOR_LOSS}) becomes inadmissible whenever the loss is a non-decreasing function of the number of type I and type II errors. It would be worth recalling in this context that, \cite{COHEN_KOL_SACK_2007}, \cite{COHEN_SACK_2005_B}, \cite{COHEN_SACK_2007} and \cite{COHEN_SACK_2008} showed that in many common applications when test statistics are dependent, typical $p$-value based stepwise testing procedures including the celebrated BH method are inadmissible with respect to the vector loss function (\ref{VECTOR_LOSS}). Consequently, such stepwise testing procedures also become inadmissible whenever the risk is a non-decreasing function of the expected number of type I and type II errors, for example, when the risk is the expected number of misclassified hypotheses. This shows a very unpleasant feature of the traditional stepwise testing procedures including the popular BH method for multiple testing under dependence. However, as we shall see later in this section that, unlike such frequentist stepwise testing procedures, our proposed BSD method is not inadmissible in the sense described above.
 
\subsection{Connection to the MRD method}
  Before we proceed further, we now establish an important connection between our proposed BSD procedure and the Maximum Residual Down (MRD) method, introduced by \cite{COHEN_SACK_XU_2009}. In particular, we show that there exists a functional relationship between the proposed BSD statistics with those of the MRD statistics. This result would be essential for showing that our proposed multiple testing procedure based on the BSD statistics will be admissible for the present testing problem. The MRD method due to \cite{COHEN_SACK_XU_2009} is based on a set of adaptively formed residuals defined as
 \begin{eqnarray}\label{MRD_STATISTICS}
 U_{tj}^{(i_1,\dots,i_{t-1})}(\vec{X}) &=& \big\{X_j - {\sigma_{(j)}^{(i_1,\dots,i_{t-1})}}^{T}{\vec{\Sigma}^{-1}_{(i_1,\dots,i_{t-1},j)}}{\vec{X}^{(i_1,\dots,i_{t-1},j)}}\big\}/ \sigma^{\frac{1}{2}}_{j\cdot(i_1,\dots,i_{t-1})},\nonumber
 \end{eqnarray}
 for $t,j=1,\dots,n$, $1\leqslant i_{1}\neq\dots\neq i_{t-1}\leqslant n$ and $i_{l}\neq j$ for all $l=1,\dots,t-1$.\newline
 
 For $1 \leqslant t \leqslant n$, we define the index $j^{\prime}_t(\vec{X})$ as
  \begin{align}\label{MRD_INDEX}
   j^{\prime}_t(\vec{X})&=\argmax_{j \in \{1,\dots,n\}\setminus\{j^{\prime}_1(\vec{X}),\dots,j^{\prime}_{t-1}(\vec{X})\}} |U^{(j^{\prime}_1(\vec{X}),\dots,j^{\prime}_{t-1}(\vec{X}))}_{tj}(\vec{X})|.
  \end{align}
 
 Given a set of positive constants $C_1 \geqslant C_2 \geqslant \dots \geqslant C_n$, the MRD method works in a step-down manner as follows:
 \begin{enumerate}
   \item At stage 1, consider the statistics $|U_{1j}(\vec{X})|$, where $j \in \{1,\dots,n\}$. If $|U_{1j^{\prime}_1(\vec{X})}(\vec{X})| \leqslant C_1, $ stop and accept all $H_{0i}$'s. Otherwise reject $H_{0j^{\prime}_1(\vec{X})}$ and continue to stage 2.\newline
 
   \item At stage 2, consider the statistics $|U^{(j^{\prime}_1(\vec{X}))}_{2j^{\prime}_1(\vec{X})}(\vec{X})|$, where $j\in\{1,\dots,n \}\setminus\{j^{\prime}_1(\vec{X})\}$. If $|U^{(j^{\prime}_1(\vec{X}))}_{2j^{\prime}_2(\vec{X})}(\vec{X})| \leqslant C_2, $ stop and accept all the remaining $H_{0i}$'s. Otherwise, reject $H_{0j^{\prime}_2(\vec{X})}$ and continue to stage 3.\newline
 
  \item In general, at stage $t$, consider the statistics $|U^{(j^{\prime}_1(\vec{X}),\dots,j^{\prime}_{t-1}(\vec{X}))}_{tj}(\vec{X})|$, where $j\in\{1,\dots,n\}\setminus\{j^{\prime}_1(\vec{X}),\dots,j^{\prime}_{t-1}(\vec{X})\}$. If $|U^{(j^{\prime}_1(\vec{X}),\dots,j^{\prime}_{t-1}(\vec{X}))}_{tj^{\prime}_t(\vec{X})}(\vec{X})| \leqslant C_t$, stop and accept all the remaining $H_{0i}$'s. Otherwise, reject $H_{0j^{\prime}_t(\vec{X})}$ and move to stage $(t+1)$.\newline
 
  \item We continue in this fashion until an acceptance occurs or there are no more null hypothesis to be tested, in which case we must stop.\newline
\end{enumerate}
 
 \begin{remark}\quad
 \begin{enumerate}
  \item Note that the indices $j_t(\vec{X}) $ and $j^{\prime}_t(\vec{X}) $ defined in (\ref{BSD_INDEX}) and (\ref{MRD_INDEX}), respectively, need not necessarily be the same.\newline

  \item The MRD procedure depends on a set of decreasing sequence of critical constants $C_1 \geqslant \dots \geqslant C_n > 0,$ choice of which are somewhat ad hoc and vary with $\vec{\Sigma}$. Performance of the MRD procedure therefore critically depends on the appropriate choice of $C_1 \geqslant \dots \geqslant C_n$, and utmost care needs to be taken while deciding over the choice of these $C_i$'s. On the other hand, in the Bayesian model selection literature, a standard practice is to compare the ratio of the posterior probabilities of two competing models with the threshold $1$ to select the model having larger posterior probability. Therefore, in applications, one may choose the thresholding constant $\delta$ used in our definition of the BSD method (see Section 5.2) to be equal to $1$ which leads to an automatic default choice of $\delta$ that works for any arbitrary $\vec{\Sigma}$.\newline
 \end{enumerate}
  \end{remark}

  Theorem \ref{THM_BSD_MRD_CONNECTION} presented below characterizes the relationship between the proposed BSD method and the MRD method due to \cite{COHEN_SACK_XU_2009}.\newline
 
  \begin{thm}\label{THM_BSD_MRD_CONNECTION}
   Under the present set-up, the BSD statistics and the MRD statistics are associated through the following functional relationship:
   \begin{eqnarray}\label{BSD_MRD_RELATION}
    S_{tj}^{(i_1,\dots,i_{t-1})}(\vec{X}) 
    &=& \frac{p}{1-p} \times \sqrt{\frac{\sigma_{j \cdot (i_1,\dots,i_{t-1})}}{V+\sigma_{j \cdot (i_1,\dots,i_{t-1})}}}\nonumber\\
    && \times\exp\bigg\{\frac{V}{2(V+\sigma_{j \cdot (i_1,\dots,i_{t-1})})}\{U_{tj}^{(i_1,\dots,i_{t-1})}(\vec{X})\}^2\bigg\}.
 \end{eqnarray}
  \end{thm} 
  
  \begin{proof}
   See Appendix.
  \end{proof}

 \subsection{Admissibility of the BSD Procedure}
 In this subsection, we show that the proposed BSD method is admissible when $\vec{X}$ is assumed to follow a multivariate normal distribution with a fixed, but unknown mean vector $\vec{\theta}$ and an arbitrary known positive definite covariance matrix $\vec{\Sigma}$. It is easy to see based on contrapositive arguments that a multiple testing procedure with respect to the vector loss function (\ref{VECTOR_LOSS}) is admissible if and only if each of the induced test procedures for the corresponding individual testing problem is admissible with respect to the standard $0-1$ loss (\ref{INDIVIDUAL_LOSS}). Hence, in order to show the admissibility property of our proposed methodology, it would suffice to establish that the corresponding induced decision for testing $H_{01}$ versus $H_{A1}$ is admissible with respect to the $0-1$ loss (\ref{INDIVIDUAL_LOSS}) under the assumed set up. As in \cite{COHEN_SACK_XU_2009}, we shall use a result due to \cite{MAT_TRU_1967} which gives a necessary and sufficient condition for the admissibility of a test of $H_{01}$ versus $H_{A1}$ when the joint distribution of $\vec{X}$ belongs to an exponential family. We emphasize in this context that although the BSD test statistics can be expressed as functions of the corresponding MRD statistics, the admissibility property of the BSD method does not follow as a direct consequence of that of the MRD procedure. It should be carefully noted that for each $t$, the functional relationship between the proposed BSD statistics and the corresponding MRD statistics in (\ref{BSD_MRD_RELATION}) involves the terms $\sigma_{j\cdot( j_1(\vec{x}),\dots,j_{t-1}(\vec{x}))}$ which depend on the set of indices $j_1(\vec{x}),\dots,j_{t-1}(\vec{x})$. Each of these indices $j_1(\vec{x}),\dots,j_{t-1}(\vec{x})$ is a function of the observed vector $\vec{x}$ and they indicate the null hypotheses those have already been rejected before the $t$-th stage. It therefore becomes necessary to understand certain behavior of these terms $\sigma_{j\cdot( j_1(\vec{x}),\dots,j_{t-1}(\vec{x}))}$ as a function of $\vec{x}$ in the decision making process when the data $\vec{x}$ is observed. Such behavior would be extremely crucial for proving the admissibility property of the proposed BSD method which will be made more precise later in this section. Moreover, our general scheme of arguments makes some of the arguments employed in the proof of the admissibility property of the MRD procedure as in \cite{COHEN_SACK_XU_2009} more explicit. In this process, we establish a more general mathematical fact which says that for the multiple testing problem (\ref{TESTING_PROBLEM}), any step-down multiple testing procedure based on a set of statistics $S_{tj}$'s such that each $S_{tj}$ is a non-decreasing function of the absolute value of the corresponding MRD statistics $U_{tj}$, is admissible with respect to the vector loss function (\ref{VECTOR_LOSS}). To the best of our knowledge this fact is hitherto unknown in the literature and extends part of the results of \cite{COHEN_SACK_XU_2009} in this context.\newline

 Let $\phi_j(\vec{x})$ denotes the test function induced by the BSD procedure for testing $ H_{0j}$ vs $H_{Aj}$ when we observe the data point $\vec{x}$. The following lemma, namely, Lemma \ref{LEM_MATH_TRUAX} is due to \cite{MAT_TRU_1967} which provides a necessary and sufficient condition for the admissibility of a testing procedure for testing $H_{01}$ versus $H_{A1}$ when $\vec{\Sigma}$ is known.\newline

 Let $\vec{Y} = \vec{\Sigma}^{-1}\vec{X}$.\newline

 \begin{lem}\label{LEM_MATH_TRUAX}
 A necessary and sufficient condition for a test $\phi(\vec{y})$ of $H_{01}$ versus $H_{A1}$ to be admissible is that, for almost every fixed $y_2, \dots,y_n$, the acceptance region of the test is an interval in $y_1$.
\end{lem}
\begin{proof}
See \cite{MAT_TRU_1967}. 
\end{proof}

Note that, for each fixed $(y_2,\dots,y_n)$, to study the test function $\phi(\vec{y})=\phi_1(\vec{x})$ as $y_1$ varies, it would be enough to consider sample points $\vec{x}+r\vec{g}$, where $\vec{g}$ is the first column of $\vec{\Sigma}$ and $r$ varies. This is true, since $\vec{y}$ is a function of $\vec{x}$, and so $\vec{y}$ evaluated at $\vec{x}+r\vec{g}$ is $\vec{\Sigma}^{-1}(\vec{x}+r\vec{g})=\vec{y}+(r,0,\dots,0)=(y_1+r,y_2,\dots,y_n)$.\newline
 
 \begin{lem}\label{LEM_PROP_OF_MRD_STAT}
 The functions $U_{tj}$ as given in equation have the following properties.\newline

 For $t \in \{1,\dots,n\}$ and for $j_1,\dots,j_{t-1} \in \{2,\dots,n\}$ with $j_i \neq j_{i'} $ for $i \neq i'$,
 \begin{eqnarray}
 U^{(j_1,\dots,j_{t-1})}_{t1}(\vec{x} + r\vec{g})=U^{(j_1,\dots,j_{t-1})}_{t1}(\vec{x})+r\sigma^{\frac{1}{2}}_{1\cdot( j_1,\dots,j_{t-1})}\nonumber
 \end{eqnarray}
 For $t \in \{1,\dots,n\}$ and for $j \in \{2,\dots,n\} \setminus \{j_1,\dots,j_{t-1}, j_1 \neq 1, \dots, j_n \neq 1 \}, $
 \begin{eqnarray}
 U^{(j_1,\dots,j_{t-1})}_{tj}(\vec{x} + r\vec{g})=U^{(j_1,\dots,j_{t-1})}_{tj}(\vec{x})\nonumber
 \end{eqnarray}
 \end{lem}
 \begin{proof}
 See the proof of Lemma 3.2 of \cite{COHEN_SACK_XU_2009}. 
 \end{proof}
 
 \begin{cor}\label{COR_FURTH_PROP_MRD_BRD_STAT}
 For any $r \in \mathbb{R}$, we have
 \begin{eqnarray}
 U_{1j}(\vec{x} + r\vec{g})=U_{1j}(\vec{x}) \mbox{ for all } j \in \{2,\dots,n\},\nonumber
\end{eqnarray}
 which, in turn, implies the following:
 \begin{eqnarray}
 S_{1j}(\vec{x} + r\vec{g})=S_{1j}(\vec{x}) \mbox{ for all } j \in \{2,\dots,n\}.\nonumber
\end{eqnarray}
 \end{cor}

\begin{remark}\label{REMARK_INCREASING_DECREASING}
 Since $\sigma_{1\cdot(j_1,\dots,j_{t-1})}>0$, it follows from Lemma \ref{LEM_PROP_OF_MRD_STAT} that for each fixed $\vec{x}\in\mathbf{R}^{n}$ and given any $(t-1)$ many indices $(j_1,\dots,j_{t-1})$, $|U^{(j_1,\dots,j_{t-1})}_{t1}(\vec{x}+r\vec{g})|$ initially decreases and then increases as $r$ increases. Also for each fixed $\vec{x}\in\mathbf{R}^{n}$ and given any $(j_1,\dots,j_{t-1})$, $U^{(j_1,\dots,j_{t-1})}_{t1}(\vec{x}+r\vec{g})$ is a strictly increasing function of $r$. Therefore, when $|U^{(j_1,\dots,j_{t-1})}_{t1}(\vec{x}+r\vec{g})|$ decreases in $r$, $U^{(j_1,\dots,j_{t-1})}_{t1}(\vec{x}+r\vec{g})$ is negative, while when $|U^{(j_1,\dots,j_{t-1})}_{t1}(\vec{x}+r\vec{g})|$ is increasing in $r$, $U^{(j_1,\dots,j_{t-1})}_{t1}(\vec{x}+r\vec{g})$ is positive. It will be seen in a short while that the preceding observation is crucial for deriving some of the important facts that follow.
\end{remark}

  In Lemma 3.2 of \cite{COHEN_SACK_XU_2009}, the term $\sigma^{1/2}_{1\cdot j_1,\dots,j_{t-1}}$ was dropped, most likely, due to some typographical error, and was not considered in subsequent theoretical analysis. However, the presence of this term in Lemma \ref{LEM_PROP_OF_MRD_STAT} of the present paper requires some careful attention since $\sigma_{1\cdot(j_1(\vec{x}),\dots,j_{t-1}(\vec{x}))}$ depends on a set of indices $j_1(\vec{x}),\dots,j_{t-1}(\vec{x})$, each of which is a function of the observed data vector $\vec{x}$. It, therefore, becomes necessary to know how this term $\sigma_{1\cdot(j_1(\vec{x}),\dots,j_{t-1}(\vec{x}))}$ behaves as $\vec{x}$ varies. A result of this kind is presented in Lemma \ref{LEM_EQUALITY_OF_INDICES} below.\newline

 Suppose $\phi_1 (\vec{x}^{*})=0$ when $\vec{x}^{*}$ is observed, that is, $\vec{x}^{*}$ is an acceptance point of $H_{01}$. Then the process must stop before $H_{01}$ gets rejected. Suppose the testing procedure stops at some stage $t$ without rejecting $H_{01}$. Let $\vec{x}^{*}+r_{0}\vec{g}$ be a point of rejection of $H_{01}$, that is,
 $\phi_1(\vec{x}^{*}+r_{0}\vec{g})=1$. Let the testing procedure rejects $H_{01}$ at some stage $t_0$ when $\vec{x}^{*}+r_{0}\vec{g}$ is observed. The next lemma gives an important identity between the set of indices $j_l(\vec{x}^{*}+r_{0}\vec{g})$ and $j_l(\vec{x}^{*})$, for $1\leqslant l\leqslant t_0-1$, which shows that these indices will remain invariant if $\min\{t,t_0\}>1$.\newline
 
 \begin{lem}\label{LEM_EQUALITY_OF_INDICES}
  Under the conditions $\phi_1 (\vec{x}^{*}) =0$ and $\phi_1 (\vec{x}^{*}+r_{0}\vec{g}) = 1$ the following holds when $t >1$ and $t_0 >1$:
 \begin{equation}
 j_l(\vec{x}^{*}+r_{0}\vec{g})= j_l(\vec{x}^{*}) \mbox{ for all } l=1,\dots,t_0 -1. \nonumber
 \end{equation}
 \end{lem}
 
\begin{proof}
 See Appendix.
\end{proof}
 
Lemma \ref{LEM_EQUALITY_OF_INDICES}, coupled with Corollary \ref{COR_FURTH_PROP_MRD_BRD_STAT}, lead to the following important result on the relation between $t_0$ and $t$ defined before.\newline

\begin{lem}\label{LEM_BSD_IMP_OBS}
Under the conditions $\phi_1 (\vec{x}^{*}) =0$ and $\phi_1 (\vec{x}^{*}+r_{0}\vec{g}) = 1,$ the BSD procedure must reject $H_{01}$ within $t$ steps when $\vec{x}^{*}+r_{0}\vec{g}$ is observed, that is $t_0 \leqslant t,$ where $t_0$ and $t$ are defined as before. 
\end{lem}

\begin{proof}
 See Appendix.
\end{proof}

\begin{lem}\label{LEM_PROP_BSD_RELN_TWO_OBS}
Suppose that for some $\vec{x}^{*}$ and $r_0 > 0,$ $\phi_1(\vec{x}^{*})=0$ and $\phi_1(\vec{x}^{*}+r_0\vec{g})=1.$ Then $\phi_1(\vec{x}^{*}+r_0\vec{g})=1$ for all $r > r_0.$
\end{lem}

\begin{proof}
 See Appendix.
\end{proof}

Using Lemma \ref{LEM_MATH_TRUAX} and Lemma \ref{LEM_PROP_BSD_RELN_TWO_OBS} above, it follows that that for testing $H_{01} \mbox { vs } H_{A1}$, the individual decision $\phi_1(\vec{X})$ induced by the BSD method would be admissible with respect to the standard $0-1$ loss (\ref{INDIVIDUAL_LOSS}). Proof that the other tests induced by the BSD method for the remaining individual testing problems will be admissible would follow analogously. Since admissibility of each individual induced decision implies the admissibility of the corresponding multiple testing procedure under the vector loss (\ref{VECTOR_LOSS}), this leads us to the desired admissibility property of our proposed testing methodology as presented in the following theorem.

\begin{thm}\label{THM_ADMISSIBILITY_BSD}
Suppose $\vec{X}\sim N_{n}(\vec{\theta},\vec{\Sigma})$, where $\vec{\theta}\in\mathbb{R}^{n}$ is unknown, but fixed and $\vec{\Sigma}$ is an $n \times n$ arbitrary but known positive definite covariance matrix. Then, for the two sided multiple testing problem (\ref{TESTING_PROBLEM}), the BSD procedure based on the statistics $S_{tj}$'s is admissible with respect to the vector loss function (\ref{VECTOR_LOSS}).
\end{thm}
 
A careful inspection of the proof of Lemma \ref{LEM_PROP_BSD_RELN_TWO_OBS} reveals that one does not need the functional form of the statistics $S_{tj}$'s for proving the desirable convexity property possessed by the BSD testing procedure. Instead, the fact that the $S_{tj}$'s are non-decreasing functions of the corresponding $|U_{tj}|$'s is all what we needed there. Hence, the corresponding arguments work equally well even if we consider any other step-down procedure which depends on a set of statistics $S_{tj}$'s, where each $S_{tj}$ is a non-decreasing function of the corresponding $|U_{tj}|$. This observation leads us to the following important theorem from which the admissibility property of the MRD testing procedure due to \cite{COHEN_SACK_XU_2009} follows immediately.\newline

\begin{thm}\label{THM_ADMISSIBILITY_GENERAL}
Under the assumptions of Theorem \ref{THM_ADMISSIBILITY_BSD}, any step-down multiple testing procedure for the two sided multiple testing problem (\ref{TESTING_PROBLEM}) based on a set of statistics $S_{tj}$'s, where each $S_{tj}$ is a non-decreasing function of the absolute value of the corresponding MRD statistics $U_{tj}$, will be admissible with respect to the vector loss function (\ref{VECTOR_LOSS}).
\end{thm}

Theorem \ref{THM_ADMISSIBILITY_GENERAL} above therefore generalizes the admissibility property of the MRD procedure to a very large collection of step-down multiple testing procedure. It should be carefully observed here that even if we consider some prior distributions to the model parameters $p$ and $V$ in (\ref{TWO_GRP_THETA_PRIOR}), the resulting version of the BSD procedure still possesses the desirable convexity property and hence it is admissible with respect to the vector loss function (\ref{VECTOR_LOSS}).


\section{Alternative Representation of the BSD Procedure}
Observe that, in order to implement the BSD procedure based on the statistics $S_{tj}$'s defined in (\ref{BSD_STAT_DEFN}), at the $t$-th stage, one needs to compute the inverses of $(n-t+1)$ many sub-matrices $\vec{\Sigma}_{(j_1(\vec{x}),\dots,j_{t-1}(\vec{x}),j)}$ and $(n-t+1)$ many ratios of determinants of the form
 \begin{eqnarray}\label{RATIO_DET}
  \frac{\big| \big(\vec{\Sigma}_{(j_1(\vec{x}),\dots,j_{t-1}(\vec{x}))}+VB_{\vec{\nu}^{(j_1(\vec{x}),\dots,j_{t-1}(\vec{x}))}:\nu_j=0,\vec{\nu}^{(j_1(\vec{x}),\dots,j_{t-1}(\vec{x}),j)}=\vec{0}}\big)\big|}{\big| \big(\vec{\Sigma}_{(j_1(\vec{x}),\dots,j_{t-1}(\vec{x}))}+VB_{\vec{\nu}^{(j_1(\vec{x}),\dots,j_{t-1}(\vec{x}))}:\nu_j=1,\vec{\nu}^{(j_1(\vec{x}),\dots,j_{t-1}(\vec{x}),j)}=\vec{0}}\big)\big|}
 \end{eqnarray}
obtained through the corresponding likelihoods, for each $t=1,\dots,n$. Both of these might be a troublesome issue from the computational viewpoint even for moderately large $n$. Moreover, while computing the ratio of determinants in (\ref{RATIO_DET}) above, it may so happen that the corresponding denominator, although positive, may be so small that the computer may report it as zero, thereby producing too many erroneous results. The latter issue may arise for certain choices of $\vec{\Sigma}$ and/or, when the number of tests $n$ is quite large, which can be avoided through the functional relationship between the BSD statistics $S_{tj}$ and the corresponding MRD statistics $U_{tj}$ as given by (\ref{BSD_MRD_RELATION}), while the challenge involved in computing the inverses of $(n-t+1)$ many sub-matrices at each step $t$ still remains. Note that one may face the same computational issue regarding the inversion of $(n-t+1)$ many sub-matrices at every step involved in the implementation the MRD procedure. In this section, we show that how the aforesaid computational issues can be overcome through an alternative way of representing the BSD statistics as well as the MRD statistics. The aforesaid alternative representation results in a remarkable computational savings and facilitates the implementation of both the BSD and MRD procedures to a large extent which will be made more precise later in this section. Towards that, we first derive some non-trivial and quite useful algebraic identities using certain results from matrix analysis. These are listed as Lemma \ref{LEM_ALT_REP_1} - Lemma \ref{LEM_ALT_REP_4} and are presented below.\newline
 
\begin{lem}\label{LEM_ALT_REP_1}
For any arbitrary variance-covariance matrix $\vec{\Sigma} $ and for any fixed $\vec{\nu}=(\nu_1,\dots,\nu_n)\in\{0,1\}^{n}$, we have the following identity:
\begin{eqnarray}
  (\vec{\Sigma}+VB_{\vec{\nu}:\nu_i=1})^{-1}=(\vec{\Sigma}+VB_{\vec{\nu}:\nu_i=0})^{-1} - \frac{V}{1+Vb_{ii}(\vec{\nu})} \vec{b_i}(\vec{\nu})\vec{b_i}(\vec{\nu})^{T}\nonumber
\end{eqnarray}
 where $\vec{b_i}(\vec{\nu})$ denotes the $i$-th column vector of the matrix $(\vec{\Sigma}+VB_{\vec{\nu}:\nu_i=0})^{-1}$ and $b_{ii}(\vec{\nu})$ is the $i$-th element of $\vec{b_i}(\vec{\nu}),$ that is, the $i$-th diagonal element of $(\vec{\Sigma}+VB_{\vec{\nu}:\nu_i=0})^{-1}$.
\end{lem}
\begin{proof}
 See Appendix.
\end{proof}

\begin{lem}\label{LEM_ALT_REP_2}
For any arbitrary positive definite covariance matrix $\vec{\Sigma} $ and for any $\vec{\nu} \in \{0,1\}^{n},$ we have the following identity:
\begin{eqnarray}
 \frac{|\vec{\Sigma}+VB_{\vec{\nu}:\nu_i=1} |}{|\vec{\Sigma}+VB_{\vec{\nu}:\nu_i=0}|}&=& 1+Vb_{ii}(\vec{\nu})\nonumber
\end{eqnarray}
where $b_{ii}(\vec{\nu})$ denotes the $i$-th diagonal element of the matrix $(\vec{\Sigma}+VB_{\vec{\nu}:\nu_i=0})^{-1}$.
\end{lem}
\begin{proof}
 See Appendix.
\end{proof}

\begin{lem}\label{LEM_ALT_REP_3}
For any arbitrary positive definite covariance matrix $\vec{\Sigma}$ and for any $\vec{\nu} \in \{0,1\}^{n}$, we have the following identity:
\begin{align}
 b_{ii}(\vec{\nu})
 &= \bigg[\sigma_{ii} - \vec{\sigma}^{T}_{(-i)}\big((\vec{\Sigma}+VB_{\vec{\nu}:\nu_i=0})_{(-i,-i)}\big)^{-1}\vec{\sigma}_{(-i)}\bigg]^{-1} \mbox{ and}\nonumber\\
 (\vec{b}_{i}(\vec{\nu}))_{(-i)}
 &= -b_{ii}(\vec{\nu})\big((\vec{\Sigma}+VB_{\vec{\nu}:\nu_i=0})_{(-i,-i)}\big)^{-1}\vec{\sigma}_{(-i)}.\nonumber
 \end{align}

Here the vector $\vec{b_i}(\vec{\nu})$ denotes the $i$-th column vector of the matrix $(\vec{\Sigma}+VB_{\vec{\nu}:\nu_i=0})^{-1}$ as already defined in Lemma \ref{LEM_ALT_REP_2} and $(\vec{b}(\vec{\nu}))_{(-i)}$ denotes the vector obtained from $\vec{b_i}(\vec{\nu})$ after removing its $i$-th coordinate. Moreover, $(\vec{\Sigma}+VB_{\vec{\nu}:\nu_i=0})_{(-i,-i)}$ is the sub-matrix obtained by removing the $i$-th row and the $i$-th column of $\vec{\Sigma}+VB_{\vec{\nu}:\nu_i=0}$.
\end{lem}
\begin{proof}
 See Appendix.
\end{proof}

\begin{lem}\label{LEM_ALT_REP_4}
 For any arbitrary positive definite covariance matrix $\vec{\Sigma}$ and for each $i=1,\dots,n$, we have
\begin{eqnarray}
\frac{f(\vec{x}|\nu_i=1,\vec{\nu}_{(-i)}=\vec{0})}{f(\vec{x}|\nu_i=0,\vec{\nu}_{(-i)}=\vec{0})}
&=& \frac{1}{\sqrt{1+Vb_{ii}}}\times\exp\bigg\{\frac{V}{2(1+Vb_{ii})} \bigg(\sum_{i=1}^{n}b_{ji}x_j\bigg)^2\bigg\} \nonumber
\end{eqnarray}
 where $\vec{b}_i=(b_{1i},\dots,b_{ni})^{T}$ denotes the $i$-th column vector of the precision matrix $\vec{\Sigma}^{-1}.$
\end{lem}
\begin{proof}
 See Appendix.
\end{proof}

Observe that Lemma \ref{LEM_ALT_REP_1} - Lemma \ref{LEM_ALT_REP_4} have been derived by taking into consideration the full data vector $\vec{X}$ and the variance-covariance matrix $\vec{\Sigma}$. However, the same results also hold true for any arbitrary partition of $\vec{X}$ and the corresponding sub-matrix of $\vec{\Sigma}$. This leads us to the following alternative representation of the BSD statistics $S_{tj}$'s as presented in Theorem \ref{THM_ALTERNATIVE_REPRESENTATION} below.\newline

\begin{thm}\label{THM_ALTERNATIVE_REPRESENTATION}
 For each step $t$ of the Bayesian Step Down procedure, the statistics $S^{(i_1,\dots,i_{t-1})}_{tj}$ can equivalently be represented as,
 \begin{eqnarray}
 S^{(i_1,\dots,i_{t-1})}_{tj} (\vec{x})
 &=& \frac{p(1-p)^{-1}}{\sqrt{1+Vb^{(i_1,\dots,i_{t-1})}_{jj}}} \times \exp\bigg\{\frac{V\bigg(\sum\limits_{k}b^{(i_1,\dots,i_{t-1})}_{kj}x^{(i_1,\dots,i_{t-1})}_k\bigg)^2}{2(1+Vb^{(i_1,\dots,i_{t-1})}_{jj})} \bigg\}\nonumber
 \end{eqnarray}
 where $\vec{b}^{(i_1,\dots,i_{t-1})}_j$ denotes the $\beta^{th}_j$ column vector of the matrix $\vec{\Sigma}^{-1}_{(i_1,\dots,i_{t-1})}$, $\beta_j$ being the position of $X_j$ among the remaining $X_i$'s after having left $X_{i_1},\dots,X_{i_{t-1}}$ and the summation within the square in the exponent of the right hand side being taken over the appropriate set of indices.
\end{thm}
\begin{proof}
 See Appendix.
\end{proof}
 
As an immediate consequence of Theorem \ref{THM_ALTERNATIVE_REPRESENTATION}, it follows that, for each $t=1,\dots,n$, at the $t$-th stage of the BSD procedure, we do not need to enumerate anymore the $(n-t+1)$ many sub-matrices of the form $\vec{\Sigma}^{-1}_{(i_1,\dots,i_{t-1})}$. Observe that, we also do not need to compute the ratios of $(n-t+1)$ many determinants described in (\ref{RATIO_DET}) at each stage $t$. Instead, Theorem \ref{THM_ALTERNATIVE_REPRESENTATION} says that now we only need to compute the inverse of $\vec{\Sigma}_{(i_1,\dots,i_{t-1})}$ whose column vectors will be used for computing the terms in the exponents of the $(n-t+1)$ statistics $S^{(i_1,\dots,i_{t-1})}_{tj}(\vec{x})$'s and the diagonal components of $\vec{\Sigma}_{(i_1,\dots,i_{t-1})}$ can be used at once instead of evaluating the $(n-t+1)$ many ratios of determinants already described before. Thus, the overall BSD procedure becomes computationally much faster compared to its original formulation. To elucidate this point, consider, for example, the situation when the testing procedure continues till the $n$-th stage. In that case, if one uses the original definition of the BSD statistics $S^{(i_1,\dots,i_{t-1})}_{tj} (\vec{x})$ as given in (\ref{BSD_STAT_DEFN}), then one has to compute the inverses of $n+(n-1)+\dots+1=\frac{n(n+1)}{2}$ many sub-matrices of $\vec{\Sigma}$ and an equal number of ratios of determinants of the form (\ref{RATIO_DET}). However, Theorem \ref{THM_ALTERNATIVE_REPRESENTATION} says that we now need to compute only the inverses of $n$ many sub-matrices of $\vec{\Sigma}$ and that's all what we need. Rest of the computation becomes immediate on modern computing platforms, even if $n$ is large. It should further be noted that using the results of \ref{LEM_ALT_REP_1} - Lemma \ref{LEM_ALT_REP_4} as well as Theorem \ref{THM_ALTERNATIVE_REPRESENTATION}, the corresponding MRD statistics $U_{tj}$'s can be rewritten as
\begin{equation}
 U^{(i_1,\dots,i_{t-1})}_{tj} (\vec{x})
 =\sum\limits_{k} b^{(i_1,\dots,i_{t-1})}_{kj}x^{(i_1,\dots,i_{t-1})}_k,\nonumber
\end{equation}
where the corresponding terms within the above summation have already been defined before. Thus, our preceding discussion clearly shows that the same computational savings can also be realized through the above alternative representation of the MRD statistics which would make it computationally much faster compared to its original formulation as in \cite{COHEN_SACK_XU_2009}.
\section{Simulations}
In this section, we present and interpret the results of our simulation study. The major objective of this study is to compare empirically the performance of the proposed BSD procedure with some well known multiple testing procedures available in the literature for the multiple testing problem (\ref{TESTING_PROBLEM_EQUIV}). Towards that, we assume the data to be generated from the $2^{n}-$component mixture of multivariate normal distributions (\ref{MARGINAL_DISTRIBUTION_X}). Our main objective would be to compare the simulated averages of the proportion of misclassified hypotheses (as estimate of the misclassification probability) and the simulated averages of the proportion of true discoveries among all the discoveries (as estimate of the power) of these testing rules with that of the BSD procedure. We consider in our study four widely used choices of $\vec{\Sigma}$ which cover various strong and weak correlation structures. Our choices of $\vec{\Sigma}$ shall be described in detail shortly. For our simulation, we generate the data by taking the sparsity parameter $p$ from the set $S=\{0.01, 0.06, 0.11, \dots, 0.96\}$, with a common lag of $0.05$ between two consecutive values of $p$. We further take $V=10$ and generate an $n$-dimensional multivariate normal vector $\vec{X}=(X_1,\dots,X_n)$ according to the mixture distribution (\ref{MARGINAL_DISTRIBUTION_X}) with $n=200$. We replicate this experiment $2000$ times for each $p$ and each $\vec{\Sigma}$. For the purpose of comparison, in our simulation study, six other multiple testing procedures apart from the proposed BSD method are considered. These are $(i)$ the step-up testing procedure due to \cite{BH1995} (abbreviated as ``BH''), $(ii)$ the step-down analogue of the BH method proposed by \cite{SKS_2002} (abbreviated as ``SDS''), $(iii)$ the fixed threshold approach due to \cite{STO_2002} (abbreviated as ``STO''), $(iv)$ the MRD testing procedure of \cite{COHEN_SACK_XU_2009}, $(v)$ the marginal testing procedure considered by \cite{XCML_2011} (abbreviated as ``MGN'') and $(vi)$ the Bayesian step-up testing procedure proposed by \cite{XCML_2011} (abbreviated as ``XCML''). It may be recalled here that the optimal Bayes testing procedure and the step-down procedure due to \cite{CHEN_SARKAR_2004} will be computationally very demanding in this context. We therefore omit them from our comparisons. \newline

We now describe the four different correlation structures used in our simulation study as follows:\newline

 \begin{enumerate}
  \item {\it Intraclass correlation structure}:\newline
  \begin{eqnarray}
   \sigma_{ij}&=& 1\mbox{ if } i=j\nonumber\\
              &=& 0.5 \mbox{ if } i\neq j.\nonumber
  \end{eqnarray}
 
 \item {\it Block dependence structure}:\newline
  \begin{eqnarray}
   \sigma_{ij} &=& 1 \mbox{ if } i = j \nonumber\\
               &=& 0.5 \mbox{ if } 1\leqslant |i-j| \leqslant 10 \nonumber\\
               &=& 0 \mbox{ otherwise.}\nonumber
  \end{eqnarray}
  
  \item {\it Autoregressive covariance structure}:\newline
  
  Let $\vec{\Sigma}$ be a Toeplitz matrix with the autoregressive correlation structure of an $AR(1)$ process with 
  \begin{eqnarray}
   \sigma_{ij} &=& 1 \mbox{ if } i=j,\nonumber\\
	       &=& 0.7^{|i-j|} \mbox{ for } i \neq j.\nonumber
  \end{eqnarray}
  
  \item {\it Short range dependence structure}:\newline
  
  Suppose $\vec{\Sigma}$ corresponds to the covariance structure involved with a moving average process with lag 1 ($MA(1)$) where
  \begin{eqnarray}
   \sigma_{ij} &=& 1\mbox{ if } i=j,\nonumber\\
	       &=& -0.3 \mbox{ if } |i-j|=1,\nonumber\\
	       &=& 0 \mbox{ otherwise}.\nonumber
  \end{eqnarray}
   \end{enumerate}
   
 Note that among the four different choices of $\vec{\Sigma}$, the first three presents strong to moderately strong correlation structures. The last one corresponds to a situation when the $X_i$'s are weakly correlated among themselves. The reason for considering such choices is to demonstrate how the performances of the testing procedures under study vary as the correlation between the test statistics gets weaker. We set $\alpha=0.1$ as our frequentist tolerance level of type I error. For the MRD testing procedure, we choose a decreasing sequence of critical constants $C_{1} \geqslant\dots\geqslant C_{n} >0$, where for the first three models we take $C_{1}=\Phi^{-1}(1-\frac{\alpha}{2n})$ and $C_{i}=0.71\Phi^{-1}(1-\frac{\alpha}{2(n-i+1)})$, $i=2,\dots,n$, while for the weak dependence model, we choose $C_{1}=\Phi^{-1}(1-\frac{\alpha}{2n})$ and $C_{i}=0.63\Phi^{-1}(1-\frac{\alpha}{2(n-i+1)})$, $i=2,\dots,n$, as prescribed in \cite{COHEN_SACK_XU_2009}. We choose the thresholding constant $\delta$ to be $1$ for the implementation of the BSD and the marginal (MGN) testing procedures.\newline
  
\begin{figure}[ht]
        \begin{center}
                \includegraphics[width=14cm,height=10cm]{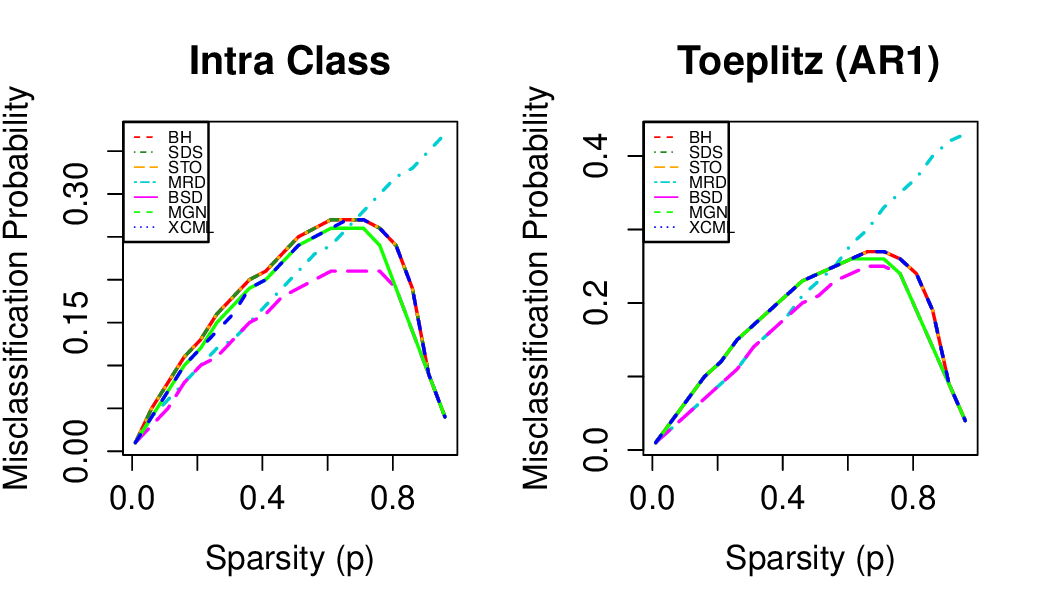}
                \caption{Comparison of estimated misclassification probabilities the intra class correlation and the AR(1) covariance structures}
                \label{FIG_MP_1}
        \end{center}
\end{figure}

\begin{figure}[ht]
        \begin{center}
                \includegraphics[width=14cm,height=10cm]{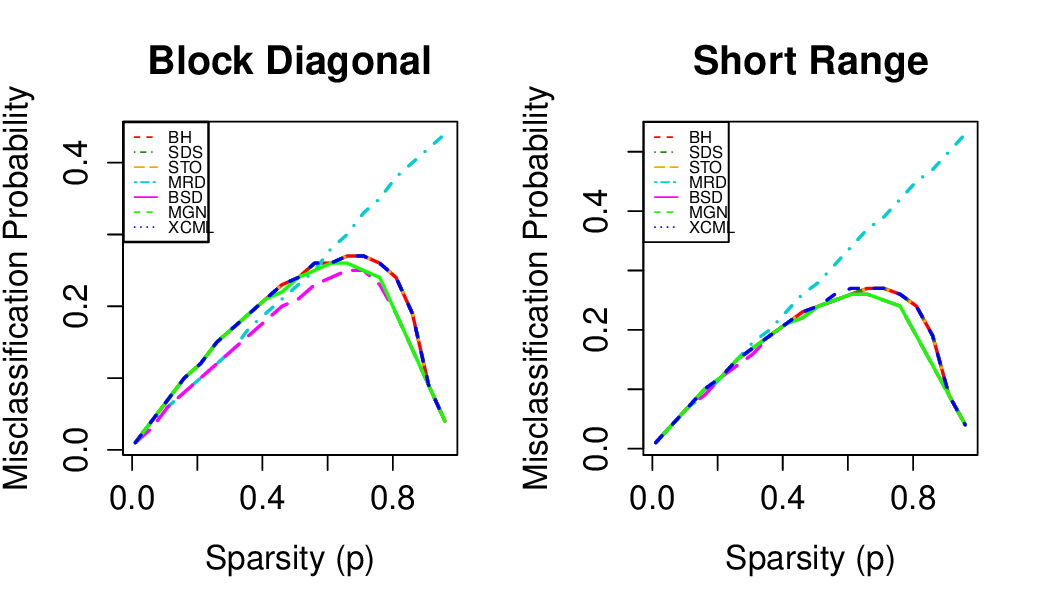}
                \caption{Comparison of estimated misclassification probabilities for the block diagonal and short range dependence (of order 2) covariance structures}
                \label{FIG_MP_2}
        \end{center}
\end{figure}

We now present the results obtained in our simulation study. Figure \ref{FIG_MP_1} presents the estimated misclassification probabilities of the various multiple testing procedures for the intraclass correlation and the $AR(1)$ autoregressive correlation structures. Figure \ref{FIG_MP_2} presents the corresponding probabilities of misclassification under the block dependence and the short range dependence models. The first thing to be noted from figures \ref{FIG_MP_1} and \ref{FIG_MP_2} is that, the BSD procedure uniformly dominates all other procedures in terms of the misclassification probability across the entire range of the sparsity parameter $p$ and for all choices of $\vec{\Sigma}$ considered in our study. In cases, where the correlations among the $X_{i}$'s are strong or moderately strong, the estimated misclassification probability corresponding to the proposed BSD procedure is considerably smaller as compared to the other testing procedures, over a wide range of the sparsity parameter $p$, particularly for values of $p$ smaller than $0.7$. However, for values of $p$ beyond this range, the performance of the MGN procedure almost coincides with that of the BSD procedure. A similar phenomena can also be observed for other testing procedures except the MRD method. The MRD method shows a significantly different behavior in terms of the estimated misclassification probability. For values of $p$ smaller than $0.4$, the performance of the MRD method is at least as good as the BSD method, although the BSD is still marginally better. As $p$ increases further, the misclassification probability of the MRD procedure continues to increase, while that for the BSD procedure initially increases at a much slower rate, and after a while, it continually decreases as $\vec{\theta}$ becomes more and more dense. Similar phenomena have also been observed whenever the dependence between $X_i$'s is strong or moderately strong for other choices of $\vec{\Sigma}$, which is not presented in this paper for reasons of space. When the correlation between $X_i$'s is weak as in the case of short range dependence structure, all the procedures under study tend to have similar misclassification probabilities as the BSD procedure for sparse to moderately sparse situations, specifically, when $p$ is smaller than $0.35$. In this case also, the misclassification probability of the MRD procedure tends to show a steady increase as $p$ increases, while the misclassification probabilities of the other testing procedures tend to be close to that of the BSD method.\newline


\begin{figure}[ht]
        \begin{center}
                \includegraphics[width=14cm,height=10cm]{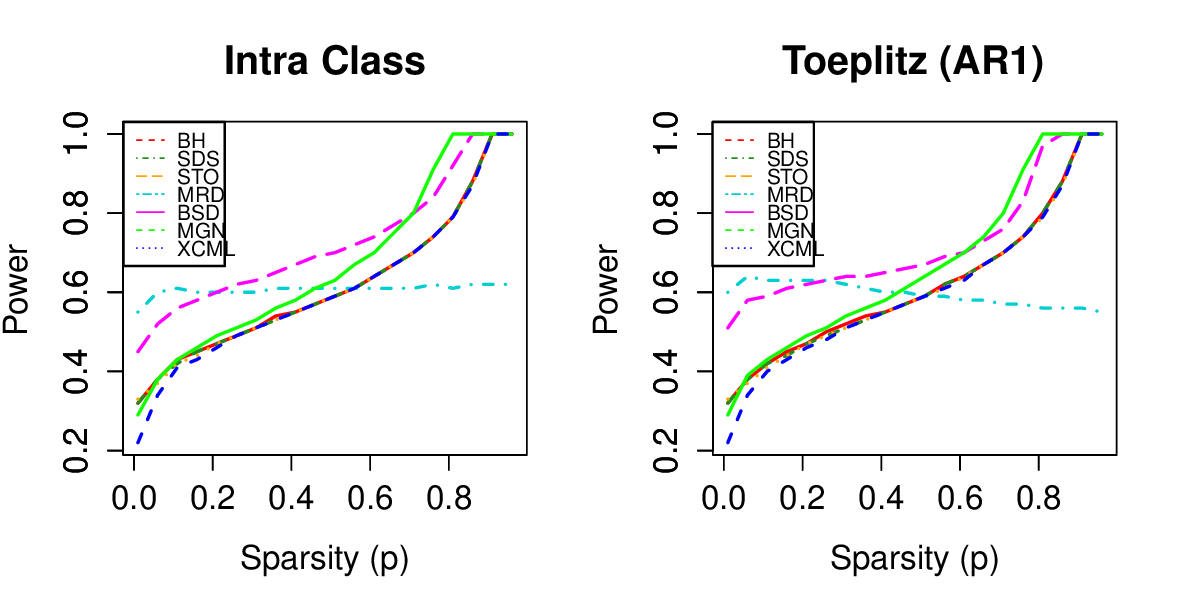}
                \caption{Comparison of estimated powers for the intra class correlation and the AR(1) covariance structures}
                \label{FIG_POWER_1}
        \end{center}
\end{figure}

\begin{figure}[ht]
        \begin{center}
                \includegraphics[width=14cm,height=10cm]{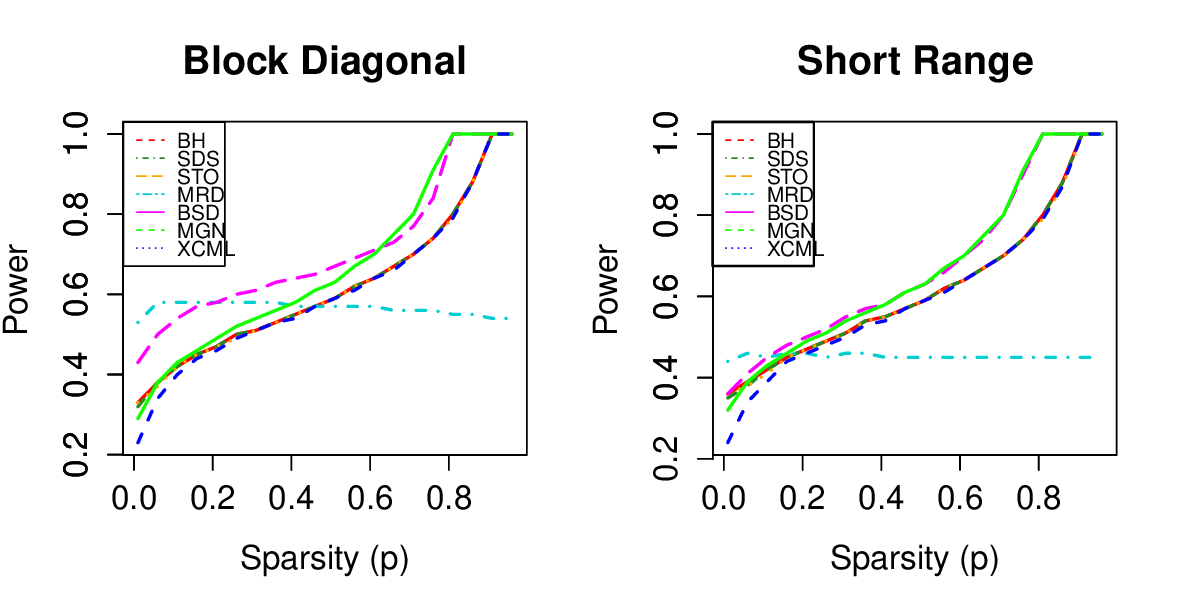}
                \caption{Comparison of estimated powers for the block diagonal and short range dependence (of order 2) covariance structures}
                \label{FIG_POWER_2}
        \end{center}
\end{figure}

We next compare the powers of the multiple testing procedures under study together with their estimated FDRs for the multiple testing problem (\ref{TESTING_PROBLEM_EQUIV}). Figures \ref{FIG_POWER_1} through \ref{FIG_FDR_2} below present the estimated powers and the false discovery rates of the multiple testing procedures under study for each value of the sparsity parameter $p$ in the set $S$ under different choices of $\vec{\Sigma}$. Figures \ref{FIG_POWER_1} and \ref{FIG_POWER_2} clearly demonstrate that when the $X_i$'s have a strong to moderately strong association among themselves, the proposed BSD procedure has a significantly larger power as compared to its competitors over a wide range of values of the sparsity parameter $p$, covering both the sparse to moderately non-sparse cases. However, as $p$ becomes larger, the marginal oracle testing procedure (MGN) tends to yield a larger power compared to the proposed BSD procedure. But this comes at the expense of considerably larger number of false discoveries made by the MGN method compared to the BSD procedure which is evident from figures \ref{FIG_FDR_1} and \ref{FIG_FDR_2}. The MRD procedure again has a different behavior in terms of its power. For strong to moderately strong correlations among the $X_{i}$'s and values of $p$ smaller than $0.2$, the MRD method yields a higher power compared to the proposed BSD procedure. This comes at the cost producing a significantly larger number of false discoveries made by the MRD method as compared to the BSD procedure, see figures \ref{FIG_POWER_1} and \ref{FIG_POWER_2}. Even when $p$ gets larger, power of the MRD method remains more or less the same and it tends to be more and more conservative as $p$ gradually approaches $1$. A different story emerges from the results under the short range dependence covariance structure. In this case, the marginal testing procedure (MGN) tends to yield a power close to the BSD method.\newline

It should be noted that the Bayesian testing procedures like the BSD method or even the optimal Bayes procedure, are typically not aimed at controlling some specific kind of type I error measure such as the FDR. In our present simulation study, we have taken $\alpha$ (the frequentist tolerance level of type I error) to be $0.1$. When the correlation among $X_i$'s are strong or moderately strong, the FDR of the proposed BSD method is also controlled at this level, provided $\vec{\theta}$ is sparse, that is, when $p$ is small. However, the same may not be true when $X_i$'s are weakly correlated among themselves. \newline
  
\begin{figure}[ht]
        \begin{center}
                \includegraphics[width=14cm,height=10cm]{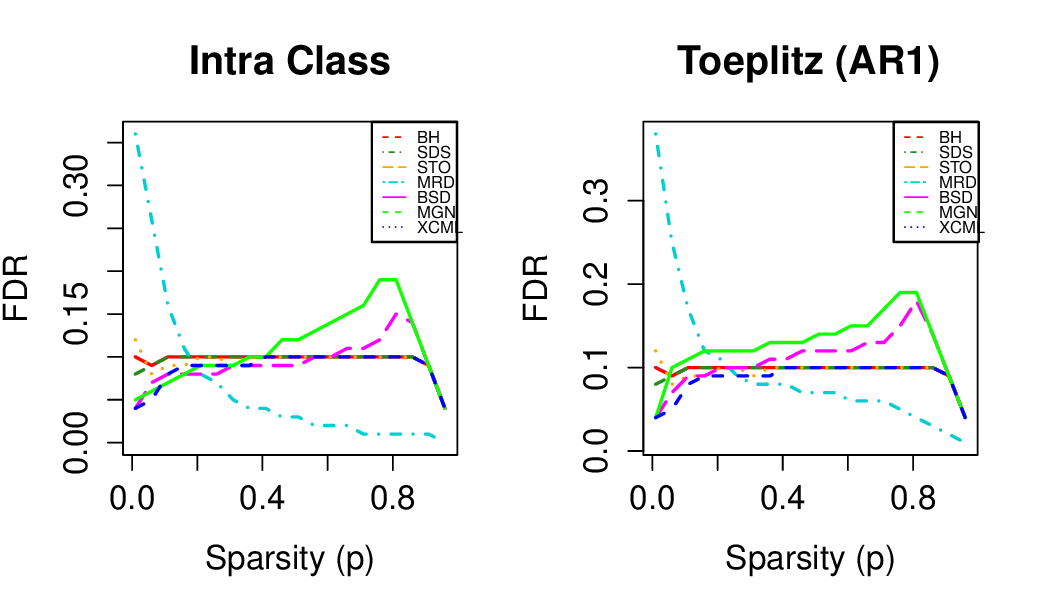}
                \caption{Comparison of estimated FDR for the intra class correlation and the AR(1) covariance structures}
                \label{FIG_FDR_1}
        \end{center}
\end{figure}

In summary, it may be said that when the association between $X_i$'s is strong or even moderately strong, the proposed BSD method, tends to outperform procedures like the BH method, the SDS method, the STO method, the MGN method and the XCML method both in terms of the estimated misclassification probability and power for a wide range of sparsity level. It also shows a decent FDR controlling property when the underlying mean vector is sparse and $X_i$'s have a strong or moderately strong association among themselves. We feel that the reasons behind the good performance of the BSD method and the MRD method can satisfactorily be explained by the following facts. First, the marginal testing procedures do not take into account the correlation among the $X_i$'s, whereas procedures like the BSD and the MRD methods fully utilizes the dependence among $X_i$'s at each step. This clearly demonstrates the effect of taking correlations into account for developing a multiple testing procedure when test statistics are correlated. Second, as already shown by \cite{COHEN_KOL_SACK_2007}, \cite{COHEN_SACK_2005_B}, \cite{COHEN_SACK_2007}, \cite{COHEN_SACK_2008} and \cite{COHEN_SACK_XU_2009} that in simultaneous testing problems involving dependent normal means, typical stepwise testing procedures are inadmissible with respect to the vector loss function, and hence, with respect to the additive loss function as well. In our context, the estimated misclassification probability is proportional to the simulated average of the total number of misclassified hypotheses which is nothing but the usual additive $0-1$ loss function within our chosen two-groups formulation. In this sense, the present simulation study is in concordance with the theoretical findings of the aforesaid papers. However, as the correlation becomes weaker and weaker, these differences tend to fade way which is understandable since in case of {\it zero} correlation (that is, under the assumption of {\it independence}), it is known that procedures like the BH method is asymptotically Bayes optimal under sparsity (ABOS) as shown in \cite{BCFG2011} and the proposed BSD method is simply the optimal Bayes decision rule in that context. The above discussion suggests that except under certain asymptotically vanishing correlation structures, the conjecture made in \cite{BCFG2011} regarding such asymptotic Bayes optimality property of the BH procedure in sparse problems and under general dependence structures, is not likely to be true.
\begin{figure}[ht]
        \begin{center}
                \includegraphics[width=14cm,height=10cm]{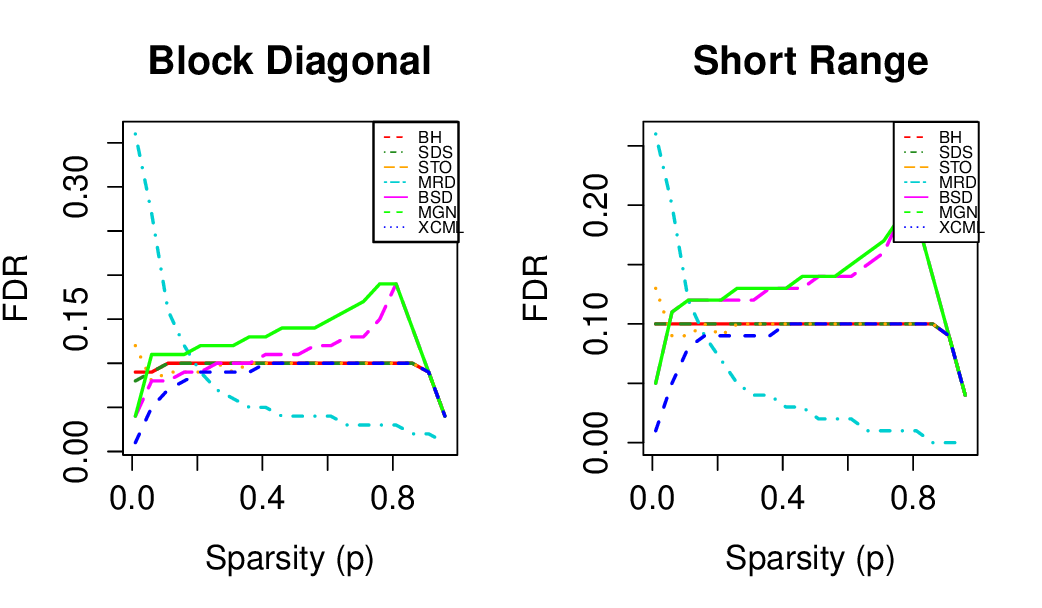}
                \caption{Comparison of estimated FDR for the block diagonal and short range dependence (of order 2) covariance structures}
                \label{FIG_FDR_2}
        \end{center}
\end{figure}

\section{Concluding Remarks}
We considered in this paper, the problem of simultaneous significance testing of the individual components of a multivariate normal mean vector when the underlying covariance matrix is assumed to be known, but arbitrary. We took a Bayesian approach where a two-component point mass mixture prior was used within a hierarchical Bayes framework to model the unknown means. Under this set up, we proposed a novel Bayesian testing procedure that works in a step-down manner and is referred to as the Bayesian Step-down (BSD) procedure. The proposed Bayesian step-down procedure is easy to implement and can fully incorporate the dependence between the test statistics at every stage unlike many other popular multiple testing approaches. Moreover, the proposed methodology provides a generic multiple testing algorithm which can be applied even for the non-normal models such as multivariate-$t$. We also established a formal decision theoretic justification in favor of our proposed testing procedure when the test statistics are assumed to come from a multivariate normal distribution with an unknown but fixed mean vector and a known positive definite covariance matrix. In particular, we employed a general technique invoking some novel arguments which shows that the proposed BSD procedure possesses a certain convexity property which is both necessary and sufficient for a multiple testing procedure to be admissible with respect to the vector loss function (\ref{VECTOR_LOSS}) for the testing problem (\ref{TESTING_PROBLEM}). As a matter of fact, using our general scheme of arguments it turned out that, any step-down multiple testing procedure based on a set of statistics which are non-decreasing functions of the absolute values of the corresponding MRD statistics, will also be admissible under the vector loss function (\ref{VECTOR_LOSS}). To the best of our knowledge, the aforesaid fact is new in the multiple testing literature and extends the results of \cite{COHEN_SACK_XU_2009} on the admissibility property of their proposed MRD method. We established an alternative representation of the proposed test statistics which leads to a great amount of computational savings for the implementation of our proposed methodology. We also demonstrated through extensive simulation study that, for various forms of dependence and across a wide range of sparsity levels, the proposed testing procedure compares quite favorably with several existing multiple testing procedures available in the literature in terms of overall misclassification probability and power.\newline

It is important to note that, implementation of the BSD method in real life applications requires the knowledge of the proportion of true alternatives $p$ and the variance $V$ of the distribution of the non-null $\theta_i$'s which may not always be known in practice. One natural approach in such cases would be to to assign some appropriate hyperpriors to $p$ and $V$, and subsequently employing some efficient Markov Chain Monte Carlo (MCMC) algorithm to find estimates of the posterior probabilities used to define the BSD statistics in (\ref{BSD_STAT_DEFN}). An alternative approach in this context could as well be the use of an empirical Bayes version of the BSD procedure by simply plugging the full Bayes estimates of $p$ and $V$ into the functional relation (\ref{BSD_MRD_RELATION}) of Theorem \ref{THM_BSD_MRD_CONNECTION}. It should, however, be noted that use of an MCMC algorithm in the present case may lead to a daunting computational task since it requires exploration over an enormously large model space of size $2^{n}$. As a result, the corresponding MCMC algorithm may suffer from very low transitional probabilities to move from one model to another and become too slow to converge. However, the difficulty concerning the implementation of such MCMC algorithms can easily be bypassed by using the functional relation (\ref{BSD_MRD_RELATION}) as follows. Let $\widehat{p}$ and $\widehat{V}$ be some ``good'' estimates of $p$ and $V$, respectively, obtained through some empirical Bayes approach or otherwise. Then, by simply plugging those estimates in (\ref{BSD_MRD_RELATION}), one can directly enumerate the BSD statistics, and thus avoid the need for MCMC-type computations for the present testing problem. The question that naturally arises then is what could possibly be the appropriate choices of $\widehat{p}$ and $\widehat{V}$ in this context.\newline


Estimation of the theoretical proportion of true alternatives or the proportion of non-null effects $p$ has been so far another topic of intense research in the multiple testing literature. There are contexts such as in astronomy where one might be more interested in estimating the proportion of true signals rather than identifying them individually. Moreover, by incorporating such estimate, the efficiency of traditional FWER or FDR controlling procedures like the BH method, can greatly be improved in terms of power. Several strategies have been proposed in the literature towards estimation of this proportion $p$. Some important early references in this regard include \cite{BH2000}, \cite{Efron2004}, \cite{ETST_2001}, \cite{GW_2004}, \cite{MR2006}, \cite{STO_2002} and \cite{STS_2004}, among others. However, the corresponding estimates of $p$ proposed in these works are, in general, inconsistent and tend to be conservative in nature. A major theoretical breakthrough in this direction was made in \cite{JIN_2006} when the underlying test statistics are assumed to be independent and identically distributed ({\it i.i.d.}) random variables generated from a two-component Gaussian mixture model. \cite{JIN_2006} proposed to estimate the corresponding mixing proportion $p$ by exploiting certain concepts from Fourier analysis. His proposed estimator is based on the central idea of approximating, what he called {\it the underlying characteristic function}, by the corresponding empirical characteristic function when the null parameter values are identical or {\it homogeneous}. He showed that the aforesaid estimator of $p$ is uniformly consistent over a large parameter space. Details of the construction of such estimates can be found in \cite{JIN_2006} and \cite{JIN_2008}. \cite{JIN_CAI_2007} extended these ideas to obtain consistent estimators of the null parameters values along with the proportion of non-null effects when the null distributions are assumed to be unknown and the null parameters are {\it heterogeneous}. Their estimators were shown to be consistent over a large parameter space and also in situations when the test statistics exhibit certain forms of dependence such as $\alpha$-mixing and short range dependence. The aforesaid estimators, though consistent, fail to attain any optimal rate of convergence. \cite{CAI_JIN_2010} considered the problem of finding consistent estimators of the null density and the proportion of non-null effects $p$ which attain the corresponding minimax error rates (with respect to appropriately chosen loss functions) up to some multiplicative factors within an {\it i.i.d.} Gaussian mixture model framework. For any fixed $\gamma \in (0,1/2)$, they proposed the following estimator of $p$, given by
 \begin{equation}\label{p_hat_definition}
  \widehat{p}(\gamma)= 1- \frac{1}{n^{1-\gamma}} \sum_{j=1}^{n}\cos\big(\sqrt{2\gamma\log n}X_j\big).
 \end{equation}

 \cite{CAI_JIN_2010} showed that $\widehat{p}(\gamma)$ defined in (\ref{p_hat_definition}) becomes minimax rate optimal when the parameter $p$ is not too small compared to the total number of tests $n$ and ``{\it vanishes asymptotically}'' as $n$ grows to infinity. \cite{CAI_JIN_2010} conjectured that the above estimator of $p$ in (\ref{p_hat_definition}) will remain consistent under certain forms of weakly correlated structures. Motivated by their work, we considered the estimator $\widehat{p}(\gamma)$ defined in (\ref{p_hat_definition}) for estimating the mixing proportion $p$ in our context. It is found that the conjecture of \cite{CAI_JIN_2010} is indeed affirmative in the sense that their estimator remains consistent under certain weakly correlated structures, such as, finite block dependence, short range dependence, certain intraclass correlation model where the common correlation coefficient goes to zero at an appropriate rate as the number of tests $n$ grows to infinity, and also in situations when $p$ is moderately sparse. Moreover, we also considered a moment-based estimate of $V$ which depends on $\widehat{p}(\gamma)$ defined above and found that it consistently estimates the variance $V$ of the non-zero $\theta_{i}$'s under certain weak correlation structures. However, in our simulation studies (which are not reported in this paper), we did not find any significant difference in the performance of the proposed BSD method as compared to its competitors under such weak correlation structures. Hence, we prefer not to report these results in this paper. It would be interesting to see whether the aforesaid results can be generalized to stronger forms of dependence or whether their exist some other estimators of $p$ and $V$ which can yield better performances both in terms of theory and simulations. Another very interesting problem would be to investigate theoretically whether the BSD and the MRD methods continue to have the admissibility property if we replace the vector loss function (\ref{VECTOR_LOSS}) with the sum of the individual losses (\ref{INDIVIDUAL_LOSS}) as proposed in \cite{LEHMANN1957a} and \cite{LEHMANN1957b} for multiple testing problems. We leave these issues as important problems for future research and hope to report them elsewhere.
\appendix  
\section{Appendix}
\label{app}
Proofs of some of the results of this paper make use of the following important results from theory of matrices. Of them the first one is the celebrated Sherman$-$Morrison$-$Woodbury (SMW) identity, while the other one provides an important formula for obtaining the inverse of a $2\times 2$ partitioned matrix. See, for example, \cite{SANT_WILL_NOTZ} among many other sources.\newline

\begin{lem}\label{SWM_IDENTITY}
  Suppose that B is any $n \times n$ nonsingular matrix, C is a $r \times r$ non-singular matrix, and A is an arbitrary $n \times r$ matrix such that $(A^{T}B^{-1}A+C)^{-1}$ is nonsingular. Then $(B+A^{T}C^{-1}A)$ is $n \times n$ non-singular with inverse given by,
  \begin{equation}
   (B+A^{T}C^{-1}A)^{-1} = B^{-1} - B^{-1}A(A^{T}B^{-1}A+C)^{-1}A^{T}B^{-1}.\nonumber
  \end{equation}   
\end{lem}

\begin{lem}\label{INV_PART_MAT}
Suppose that $B$ is a $n \times n$ non-singular matrix and
 \begin{eqnarray}
 T
 &=&
 \begin{pmatrix}
 D & A^{T}\\ 
 A & B
 \end{pmatrix},\nonumber
 \end{eqnarray}
where $D$ is $m\times m$ and $A$ is $n\times m$. Then $T$ is non-singular if and only if
\begin{equation}
 Q=D-A^{T}B^{-1}A\nonumber
\end{equation}
is non-singular. In this case, $T^{-1}$ is given by
 \begin{eqnarray}
 T^{-1}
 &=&
 \begin{pmatrix}
 Q^{-1} & -Q^{-1}A^{T}B^{-1}\\ 
 -B^{-1}AQ^{-1} & B^{-1}+B^{-1}AQ^{-1}A^{T}B^{-1}
 \end{pmatrix}.\nonumber
 \end{eqnarray}
\end{lem}

\begin{flushleft}
  \textbf{Proof of Theorem \ref{THM_BSD_MRD_CONNECTION}}
\end{flushleft}
  \begin{proof}
    Observe that one can write each test statistic $S_{tj}^{(i_1,\dots,i_{t-1})}$ as
  \begin{eqnarray}\label{MRD_BSD_RELATION_1}
  S_{tj}^{(i_1,\dots,i_{t-1})}(\vec{X})
  &=& \frac{\pi(\nu_j=1,\vec{\nu}^{(i_1,\dots,i_{t-1}, j)}=\vec{0}) f(\vec{X}^{(i_1,\dots,i_{t-1})}|\nu_j=1,\vec{\nu}^{(i_1,\dots,i_{t-1}, j)}=\vec{0})}{\pi(\nu_j=0,\vec{\nu}^{(i_1,\dots,i_{t-1}, j)}=\vec{0}) f(\vec{X}^{(i_1,\dots,i_{t-1})}|\nu_j=0,\vec{\nu}^{(i_1,\dots,i_{t-1}, j)}=\vec{0})}\nonumber\\
  &=& \frac{p}{1-p}\times \frac{f(\vec{X}^{(i_1,\dots,i_{t-1})}|\nu_j=1,\vec{\nu}^{(i_1,\dots,i_{t-1}, j)}=\vec{0})}{f(\vec{X}^{(i_1,\dots,i_{t-1})}|\nu_j=0,\vec{\nu}^{(i_1,\dots,i_{t-1}, j)}=\vec{0})}\cdot
  \end{eqnarray}
  
  Let us write
  \begin{equation}\label{MRD_BSD_RELATION_2} 
  \left.\begin{aligned}
  \vec{\Sigma}_{0,j}^{(i_1,\dots,i_{t-1})} &= \mbox{ } \vec{\Sigma}_{(i_1,\dots,i_{t-1})}+V diag(\nu_j=0,\vec{\nu}^{(i_1,\dots,i_{t-1}, j)}=\vec{0})\\
  \vec{\Sigma}_{1,j}^{(i_1,\dots,i_{t-1})}&=\vec{\Sigma}_{(i_1,\dots,i_{t-1})}+V diag(\nu_j=1,\vec{\nu}^{(i_1,\dots,i_{t-1}, j)}=\vec{0})
  \end{aligned}\right\}.
  \end{equation}
%
  
  It is important to note that for any $t=1,\dots,n$ and for any $j=1,\dots,n,$ where $i_l \neq j$ for all l, we have the following:
  \begin{eqnarray}\label{MRD_BSD_RELATION_3}
  \vec{\Sigma}_{0,j}^{(i_1,\dots,i_{t-1})}
  &=&\big(\vec{\Sigma}_{(i_1,\dots,i_{t-1})}+V diag(\nu_j=0,\vec{\nu}^{(i_1,\dots,i_{t-1}, j)}=\vec{0})\big)_{(-j,-j)}\nonumber\\
  &=&\big(\vec{\Sigma}_{(i_1,\dots,i_{t-1})}\big)_{(-j,-j)} \nonumber\\
  &=&\vec{\Sigma}_{(i_1,\dots,i_t,j)}
  \end{eqnarray}
  and
  \begin{eqnarray}\label{MRD_BSD_RELATION_4}
  \vec{\Sigma}_{1,j}^{(i_1,\dots,i_{t-1})}
  &=&\big(\vec{\Sigma}_{(i_1,\dots,i_{t-1})}+V diag(\nu_j=1,\vec{\nu}^{(i_1,\dots,i_{t-1}, j)}=\vec{0})\big)_{(-j,-j)}\nonumber\\
  &=&\big(\vec{\Sigma}_{(i_1,\dots,i_{t-1})}\big)_{(-j,-j)} \nonumber\\
  &=&\vec{\Sigma}_{(i_1,\dots,i_t,j)} 
  \end{eqnarray}
 where $A_{(-j,-j)} $ denotes the sub-matrix of a matrix $A$ obtained after removing its $j^{th}$ row and $j^{th}$ column.  \newline
 
 Using (\ref{MRD_BSD_RELATION_3}) and (\ref{MRD_BSD_RELATION_4}), it therefore follows that for any $t=1,\dots,n$ and for any $j=1,\dots,n,$ with $i_l \neq j$ for all $l$, we have
  \begin{eqnarray}\label{MRD_BSD_RELATION_5}
  \vec{\Sigma}_{0,j}^{(i_1,\dots,i_{t-1})} = \vec{\Sigma}_{1,j}^{(i_1,\dots,i_{t-1})}.
  \end{eqnarray}
 
 Now, $f(\vec{X}^{(i_1,\dots,i_{t-1})}|\nu_j=0,\vec{\nu}^{(i_1,\dots,i_{t-1}, j)}=\vec{0})$ corresponds to the probability density function of a $N(\vec{0},\vec{\Sigma}_{0,j}^{(i_1,\dots,i_{t-1})})$ distribution. Therefore using equation (2.4) one can write,  
    \begin{eqnarray}\label{MRD_BSD_RELATION_6}
    f(\vec{X}^{(i_1,\dots,i_{t-1})}|\nu_j=0,\vec{\nu}^{(i_1,\dots,i_{t-1}, j)}=\vec{0})
    &=& N(u_{j\cdot(i_1 \dots, i_{t-1})}(\vec{X}),\sigma_{j\cdot(i_1,\dots,i_{t-1})})(X_j)\nonumber\\
    &\times& N(\vec{0},\vec{\Sigma}_{(i_1,\dots,i_t,j)})(\vec{X}^{(i_1,\dots,i_{t-1},j)})
    \end{eqnarray}
    where $u_{j\cdot(i_1 \dots, i_{t-1})}(\vec{X})={\vec{\sigma}_{(j)}^{(i_1 \dots, i_{t-1})}}^{T}{\vec{\Sigma}^{-1}_{(i_1,\dots,i_{t-1},j)}}{\vec{X}^{(i_1,\dots,i_{t-1},j)}}$ and the term on the right hand side of (\ref{MRD_BSD_RELATION_6}) denote the probability densities of the corresponding normal distributions evaluated at the appropriate points.\newline
 
   In a similar way we can write the following:
    \begin{eqnarray}\label{MRD_BSD_RELATION_7}
    f(\vec{X}^{(i_1,\dots,i_{t-1})}|\nu_j=1,\vec{\nu}^{(i_1,\dots,i_{t-1}, j)}=\vec{0})
    &=& N(u_{j\cdot(i_1 \dots, i_{t-1})}(\vec{X}),V+\sigma_{j\cdot(i_1,\dots,i_{t-1})})(X_j)\nonumber\\
    &\times& N(\vec{0},\vec{\Sigma}_{(i_1,\dots,i_t,j)})(\vec{X}^{(i_1,\dots,i_{t-1},j)}).
    \end{eqnarray}
    
 On combining equations (\ref{MRD_BSD_RELATION_1}), (\ref{MRD_BSD_RELATION_3}) - (\ref{MRD_BSD_RELATION_7}), together with some subsequent straightforward calculus, leads to the proof of Theorem \ref{THM_BSD_MRD_CONNECTION}.
    \end{proof}  

\begin{flushleft}    
\textbf{Proof of Lemma \ref{LEM_EQUALITY_OF_INDICES}}
\end{flushleft}

\begin{proof}
First observe that since both $t >1$ and $t_0 >1$, one must have $j_1(\vec{x}^{*}+r_{0}\vec{g}) \neq 1$ and $j_1(\vec{x}^{*}) \neq 1$. Then using the observation made in Remark \ref{REMARK_INCREASING_DECREASING} we obtain,
\begin{eqnarray}\label{LEM_EQUALITY_OF_INDICES_EQ1}
j_1(\vec{x}^{*}+r_{0}\vec{g})
&=& \argmax_{j \in \{2,\dots,n\}} S_{1j}(\vec{x}^{*}+r_{0}\vec{g})\nonumber\\
&=& \argmax_{j \in \{2,\dots,n\}} S_{1j}(\vec{x}^{*})\nonumber\\
&=& j_1(\vec{x}^{*}).
\end{eqnarray}

Now using Lemma \ref{LEM_PROP_OF_MRD_STAT} it follows that, for all $l=1,\dots,t_0-1$ with $j_l(\vec{x}^{*}+r_{0}\vec{g})\neq 1$, and for all $j\in\{1,\dots,n\}\setminus\{j_1(\vec{x}^{*}+r_{0}\vec{g}),\dots,j_{t_0-1}(\vec{x}^{*}+r_{0}\vec{g})\}$,
\begin{equation}\label{LEM_EQUALITY_OF_INDICES_EQ2}
 U_{lj}^{(j_1(\vec{x}^{*}+r_{0}\vec{g}),\dots,j_{l-1}(\vec{x}^{*}+r_{0}\vec{g}))}(\vec{x}^{*}+r_{0}\vec{g})=U_{lj}^{(j_1(\vec{x}^{*}+r_{0}\vec{g}),\dots,j_{l-1}(\vec{x}^{*}+r_{0}\vec{g}))}(\vec{x}^{*}).
\end{equation}

Therefore, we obtain for all $l=1,\dots,t_0-1$ with $j_l(\vec{x}^{*}+r_{0}\vec{g}) \neq 1$, and for all $j\in\{1,\dots,n\} \setminus \{j_1(\vec{x}^{*}+r_{0}\vec{g}),\dots,j_{t_0-1}(\vec{x}^{*}+r_{0}\vec{g})\}$,
\begin{equation}\label{LEM_EQUALITY_OF_INDICES_EQ3}
 S_{lj}^{(j_1(\vec{x}^{*}+r_{0}\vec{g}),\dots,j_{l-1}(\vec{x}^{*}+r_{0}\vec{g}))}(\vec{x}^{*}+r_{0}\vec{g})=S_{lj}^{(j_1(\vec{x}^{*}+r_{0}\vec{g}),\dots,j_{l-1}(\vec{x}^{*}+r_{0}\vec{g}))}(\vec{x}^{*}).
\end{equation}

In particular, for all $l < t_0$, 
\begin{equation}\label{LEM_EQUALITY_OF_INDICES_EQ4}
 S_{lj_l(\vec{x}^{*}+r_{0}\vec{g})}^{(j_1(\vec{x}^{*}+r_{0}\vec{g}),\dots,j_{l-1}(\vec{x}^{*}+r_{0}\vec{g}))}(\vec{x}^{*}+r_{0}\vec{g}) = S_{lj_l(\vec{x}^{*}+r_{0}\vec{g})}^{(j_1(\vec{x}^{*}+r_{0}\vec{g}),\dots,j_{l-1}(\vec{x}^{*}+r_{0}\vec{g}))}(\vec{x}^{*}).
\end{equation}

Again using Lemma \ref{LEM_PROP_OF_MRD_STAT} we have for all $l=1,\dots,t_0 -1$  with $j_l(\vec{x}^{*}+r_{0}\vec{g}) \neq 1,$ and for all $j \in \{1,\dots,n\} \setminus \{j_1(\vec{x}^{*}+r_{0}\vec{g}),\dots,j_{t-1}(\vec{x}^{*}+r_{0}\vec{g})\}$ the following:
\begin{eqnarray}
 U_{l1}^{(j_1(\vec{x}^{*}+r_{0}\vec{g}),\dots,j_{l-1}(\vec{x}^{*}+r_{0}\vec{g}))}(\vec{x}^{*} + r_{0}\vec{g})
 &=& U_{l1}^{(j_1(\vec{x}^{*}+r_{0}\vec{g}),\dots,j_{l-1}(\vec{x}^{*}+r_{0}\vec{g}))}(\vec{x}^{*})\nonumber\\
 &+& r_0\cdot \sigma_{1\cdot(j_1(\vec{x}^{*}+r_{0}\vec{g}),\dots,j_{l-1}(\vec{x}^{*}+r_{0}\vec{g}))}^{\frac{1}{2}}.\nonumber
\end{eqnarray}

The above equality implies that only the values of $S_{l1}^{(j_1(\vec{x}^{*}+r_{0}\vec{g}),\dots,j_{l-1}(\vec{x}^{*}+r_{0}\vec{g}))}(\vec{x}^{*}+r_{0}\vec{g})$ can change for $l=1,\dots,t_0-1$. Again since $H_{01}$ is rejected at the $t_0$-th stage when $\vec{x}^{*}+r_{0}\vec{g}$ is observed, for each $l=1,\dots, t_0-1$, $S_{l1}^{(j_1(\vec{x}^{*}+r_{0}\vec{g}),\dots,j_{l-1}(\vec{x}^{*}+r_{0}\vec{g}))}(\vec{x}^{*}+r_{0}\vec{g})$ cannot be the maximum of the corresponding $S_{lj}^{(j_1(\vec{x}^{*}+r_{0}\vec{g}),\dots,j_{l-1}(\vec{x}^{*}+r_{0}\vec{g}))}(\vec{x}^{*}+r_{0}\vec{g})$'s, since in that case $H_{01}$ would have been rejected before the $t_0$-th step which would be a contradiction. Using this observation and equations (\ref{LEM_EQUALITY_OF_INDICES_EQ1}), (\ref{LEM_EQUALITY_OF_INDICES_EQ3}) and (\ref{LEM_EQUALITY_OF_INDICES_EQ4}) it therefore follows that for any $1 \leqslant l \leqslant t_0-1$,
\begin{eqnarray}
j_l(\vec{x}^{*}+r_{0}\vec{g})
&=& \argmax_{j \in \{2,\dots,n\}\setminus\{j_1(\vec{x}^{*}+r_{0}\vec{g}), \dots, j_{l-1}(\vec{x}^{*}+r_{0}\vec{g})\}} S_{lj}^{(j_1(\vec{x}^{*}+r_{0}\vec{g}),\dots, j_{l-1}(\vec{x}^{*}+r_{0}\vec{g}))}(\vec{x}^{*}+r_{0}\vec{g})\nonumber\\
&=& \argmax_{j \in \{2,\dots,n\}\setminus\{j_1(\vec{x}^{*}+r_{0}\vec{g}),\dots, j_{l-1}(\vec{x}^{*}+r_{0}\vec{g})\}} S_{lj}^{(j_1(\vec{x}^{*}+r_{0}\vec{g}),\dots, j_{l-1}(\vec{x}^{*}+r_{0}\vec{g}))}(\vec{x}^{*})\nonumber\\
&=& \argmax_{j \in \{2,\dots,n\}\setminus\{j_1(\vec{x}^{*}), \dots, j_{l-1}(\vec{x}^{*})\}} S_{lj}^{(j_1(\vec{x}^{*}), \dots, j_{l-1}(\vec{x}^{*}))}(\vec{x}^{*})\nonumber\\
&=& j_l(\vec{x}^{*})\nonumber
\end{eqnarray}
This completes the proof of Lemma \ref{LEM_EQUALITY_OF_INDICES}.
\end{proof}

\begin{flushleft}
\textbf{Proof of Lemma \ref{LEM_BSD_IMP_OBS}}
\end{flushleft}
\begin{proof}
First observe that, when $t_0=1$, since $\phi_1 (\vec{x}^{*}) =0$ and $\phi_1 (\vec{x}^{*}+r_{0}\vec{g}) = 1$, one cannot have $t=1$ due to (\ref{LEM_EQUALITY_OF_INDICES_EQ1}). Therefore we must have $t>1$ when $t_0=1$. Thus the result is true when $t_0=1$. However, proof for the case when both $t>1$ and $t_0 > 1$ is non-trivial and requires a contrapositive argument and Lemma \ref{LEM_EQUALITY_OF_INDICES}. So, let us now consider the case when $t>1$ and $t_0 > 1$.\newline

Since $t>1$, we have $S_{tj_l(\vec{x}^{*})}^{(j_1(\vec{x}^{*}),\dots,j_{t-1}(\vec{x}^{*}))}(\vec{x}^{*}) > \delta$ for all $l=1,\dots, t-1$, with $j_l(\vec{x}^{*}) \neq 1$ for each $l$ and
\begin{eqnarray}
S_{tj_{t}(\vec{x}^{*})}^{(j_1(\vec{x}^{*}),\dots,j_{t-1}(\vec{x}^{*}))}(\vec{x}^{*}) \leqslant \delta \Longrightarrow S_{tj}^{(j_1(\vec{x}^{*}),\dots,j_{t-1}(\vec{x}^{*}))}(\vec{x}^{*}) \leqslant \delta \nonumber
\end{eqnarray}
for all $j \in \{1,\dots,n\} \setminus \{j_1(\vec{x}^{*}),\dots,j_{t-1}(\vec{x}^{*})\}$, with $j_l(\vec{x}^{*}) \neq 1$ for all $l \in \{1,\dots,t-1\}$.\newline

On contrary, let us now assume that $t_0>t$. Then
\begin{eqnarray}
 S_{tj_t(\vec{x}^{*}+r_{0}\vec{g})}^{(j_1(\vec{x}^{*}+r_{0}\vec{g}),\dots,j_{t-1}(\vec{x}^{*}+r_{0}\vec{g}))}(\vec{x}^{*}+r_{0}\vec{g})>\delta, 
\end{eqnarray}
otherwise the process would have stopped at stage $t$ without rejecting $H_{01}$ when $\vec{x}^{*}+r_{0}\vec{g}$ is observed, which would be a contradiction.\newline

Now by using Lemma \ref{LEM_PROP_OF_MRD_STAT} and a subsequent application of Lemma \ref{LEM_EQUALITY_OF_INDICES} it follows
\begin{eqnarray}
 U_{tj_t(\vec{x}^{*}+r_{0}\vec{g})}^{(j_1(\vec{x}^{*}+r_{0}\vec{g}),\dots,j_{t-1}(\vec{x}^{*}+r_{0}\vec{g}))}(\vec{x}^{*}+r_{0}\vec{g})
 &=& U_{tj_t(\vec{x}^{*}+r_{0}\vec{g})}^{(j_1(\vec{x}^{*}+r_{0}\vec{g}),\dots,j_{t-1}(\vec{x}^{*}+r_{0}\vec{g}))}(\vec{x}^{*})\nonumber\\
 &=& U_{tj_t(\vec{x}^{*})}^{(j_1(\vec{x}^{*}),\dots,j_{t-1}(\vec{x}^{*}))}(\vec{x}^{*}).\nonumber
\end{eqnarray}

Therefore we have
\begin{eqnarray}
 S_{tj_t(\vec{x}^{*})}^{(j_1(\vec{x}^{*}),\dots,j_{t-1}(\vec{x}^{*}))}(\vec{x}^{*})
 &=& S_{tj_t(\vec{x}^{*}+r_{0}\vec{g})}^{(j_1(\vec{x}^{*}+r_{0}\vec{g}),\dots,j_{t-1}(\vec{x}^{*}+r_{0}\vec{g}))}(\vec{x}^{*}+r_{0}\vec{g})\nonumber\\
 &>& \delta.\nonumber
\end{eqnarray}

This means that when $\vec{x}^{*}$ is observed, the testing procedure cannot stop at stage $t$ and consequently $\phi_1(\vec{x}^{*})\neq 0$, which is a contradiction. This completes the proof of Lemma \ref{LEM_BSD_IMP_OBS}.
\end{proof}

\begin{flushleft}
\textbf{Proof of Lemma \ref{LEM_PROP_BSD_RELN_TWO_OBS}}
\end{flushleft}

\begin{proof}
Let us first consider the situation when $t_0>1$.\newline

Observe that when $t_0 >1$ we have,
\begin{eqnarray}
 S_{t_{0}1}^{(j_1(\vec{x}^{*}),\dots,j_{t_{0}-1}(\vec{x}^{*}))}(\vec{x}^{*})
 &\leqslant& S_{t_{0}j_{t_{0}}(\vec{x}^{*})}^{(j_1(\vec{x}^{*}),\dots,j_{t_{0}-1}(\vec{x}^{*}))}(\vec{x}^{*}) \nonumber\\
 &=& S_{t_{0}j_{t_{0}}(\vec{x}^{*})}^{(j_1(\vec{x}^{*}),\dots,j_{t_{0}-1}(\vec{x}^{*}))}(\vec{x}^{*}+r_{0}\vec{g}) \nonumber\\
 &=& S_{t_{0}j_{t_{0}}(\vec{x}^{*})}^{(j_1(\vec{x}^{*}+r_{0}\vec{g}),\dots,j_{t_{0}-1}(\vec{x}^{*}+r_{0}\vec{g}))}(\vec{x}^{*}+r_{0}\vec{g}) \nonumber\\
 &\leqslant& S_{t_{0}1}^{(j_1(\vec{x}^{*}+r_{0}\vec{g}),\dots,j_{t_{0}-1}(\vec{x}^{*}+r_{0}\vec{g}))}(\vec{x}^{*}+r_{0}\vec{g}) \nonumber\\
 &=& S_{t_{0}1}^{(j_1(\vec{x}^{*}),\dots,j_{t_{0}-1}(\vec{x}^{*}))}(\vec{x}^{*}+r_{0}\vec{g}).\nonumber
\end{eqnarray}

Since for given $j_1,\dots,j_{t_{0}-1},$ $S_{t_{0}1}^{(j_1,\dots,j_{t_{0}-1})}$ is a strictly increasing function of
$|U_{t_{0}1}^{(j_1,\dots,j_{t_{0}-1})}|,$ it follows that
\begin{align}
|U_{t_{0}1}^{(j_1(\vec{x}^{*}),\dots,j_{t_{0}-1}(\vec{x}^{*}))}(\vec{x}^{*})|
&\leqslant |U_{t_{0}1}^{(j_1(\vec{x}^{*}),\dots,j_{t_{0}-1}(\vec{x}^{*}))}(\vec{x}^{*}+r_{0}\vec{g})|\label{LEM_PROP_OF_MRD_STAT_EQ1}
\end{align}
whence it follows from Remark \ref{REMARK_INCREASING_DECREASING} that
\begin{align}
U_{t_{0}1}^{(j_1(\vec{x}^{*}),\dots,j_{t_{0}-1}(\vec{x}^{*}))}(\vec{x}^{*}+r_{0}\vec{g})>0.\nonumber
\end{align}

But for given $(j_1,\dots,j_{t_{0}-1})$, the function $U_{t_{0}1}^{(j_1,\dots,j_{t_{0}-1})}(\vec{x}^{*}+r\vec{g})$ is strictly increasing in $r$. Hence for all $r > r_0,$ we have
\begin{eqnarray}\label{LEM_PROP_OF_MRD_STAT_EQ2}
 U_{t_{0}1}^{(j_1(\vec{x}^{*}),\dots,j_{t_{0}-1}(\vec{x}^{*}))}(\vec{x}^{*}+r\vec{g})
 &>& U_{t_{0}1}^{(j_1(\vec{x}^{*}),\dots,j_{t_{0}-1}(\vec{x}^{*}))}(\vec{x}^{*}+r_{0}\vec{g})>0.
\end{eqnarray}

We shall complete the proof now based on a contrapositive argument. Recall that, we need to show $\phi_1(\vec{x}^{*}+r\vec{g})=1$ for all $r > r_0.$ On contrary, suppose this is not true. Then there exists some $r_1 > r_0 $ such that $\phi_1(\vec{x}^{*}+r_1\vec{g})=0.$ Let $t_1$ denote the step at which the testing procedure must stop without rejecting $H_{01}$ when $\vec{x}^{*}+r_1\vec{g}$ is observed. Then using Lemma \ref{LEM_BSD_IMP_OBS} we have $t_0 \leqslant t_1$. Since $t_0>1$, using Lemma \ref{LEM_EQUALITY_OF_INDICES} it follows
\begin{eqnarray}
 j_l(\vec{x}^{*}+r_1\vec{g})
 &=& j_l(\vec{x}^{*}+r_{0}\vec{g})\nonumber\\
 &=& j_l(\vec{x}^{*})\nonumber
\end{eqnarray}
for all $l=1,\dots,t_0-1$.\newline

Again, replacing $\vec{x}^{*}$ by $\vec{x}^{*}+r_1\vec{g}$, and applying the preceding arguments, from (\ref{LEM_PROP_OF_MRD_STAT_EQ1}) we obtain 
\begin{eqnarray}
|U_{t_{0}1}^{(j_1(\vec{x}^{*}),\dots,j_{t_{0}-1}(\vec{x}^{*}))}(\vec{x}^{*}+r_1\vec{g})|
&\leqslant& |U_{t_{0}1}^{(j_1(\vec{x}^{*}),\dots,j_{t_{0}-1}(\vec{x}^{*}))}(\vec{x}^{*}+r_{0}\vec{g})|\nonumber
\end{eqnarray}
which contradicts (\ref{LEM_PROP_OF_MRD_STAT_EQ2}). Therefore one must have $\phi_1(\vec{x}^{*}+r\vec{g})=1 $ for all $r>r_0$, when $t_0>1$.\newline

Next observe that when $t_0=1$, since $S_{1j}(\vec{x}^{*}+r_{0}\vec{g})=S_{1j}(\vec{x}^{*})$ for all $j \in \{2,\dots,n\}$, one must have $S_{11}(\vec{x}^{*}) < S_{11}(\vec{x}^{*}+r_{0}\vec{g})$. Therefore, using exactly the same arguments as before we have,
\begin{eqnarray}
 U_{11}(\vec{x}^{*}+r\vec{g}) &>& U_{11}(\vec{x}^{*}+r_{0}\vec{g}) \mbox{ }>\mbox{ } 0 \mbox{ for all } r>r_0.\nonumber 
\end{eqnarray}

Therefore, using the preceding arguments together with Corollary \ref{COR_FURTH_PROP_MRD_BRD_STAT}, for all $r > r_0$ we obtain
\begin{eqnarray}
 S_{11}(\vec{x}^{*}+r\vec{g})
 &>& S_{11}(\vec{x}^{*}+r_{0}\vec{g})\nonumber\\
 &\geqslant& S_{1j}(\vec{x}^{*}+r_{0}\vec{g})\mbox{ for all }j\in\{2,\dots,n\}\mbox{ }\big[\mbox{since }t_0=1\big]\nonumber\\
 &=& S_{1j}(\vec{x}^{*}) \mbox{ for all } j\in \{2,\dots,n\}\nonumber\\
 &=& S_{1j}(\vec{x}^{*}+r\vec{g}) \mbox{ for all } j\in \{2,\dots,n\}.\nonumber
\end{eqnarray}

This implies that every $\vec{x}^{*}+r\vec{g}$ will be a point of rejection for $H_{01}$ for all $r > r_0$, that is, $\phi_1(\vec{x}^{*}+r\vec{g})=1 $ for all $r > r_0$ when $t_0=1$. This completes the proof of Lemma \ref{LEM_PROP_BSD_RELN_TWO_OBS}.
\end{proof}


\begin{flushleft}
  \textbf{Proof of Lemma \ref{LEM_ALT_REP_1}}
\end{flushleft}

\begin{proof}
Let us write $B_{1,\vec{\nu}}=\vec{\Sigma}+VB_{\vec{\nu}:\nu_i=1}$ and $B_{0,\vec{\nu}}=\vec{\Sigma}+VB_{\vec{\nu}:\nu_i=0}$. Then letting $\vec{e_i}$ to be the unit vector with the $i$-th component unity, we get
\begin{eqnarray}\label{LEM_ALT_REP_1_EQ1}
 B_{1,\vec{\nu}}
 &=& (\vec{\Sigma}+VB_{\vec{\nu}:\nu_i=0}) + Vdiag(\vec{e_i})\nonumber\\
 &=& B_{0,\vec{\nu}} + Vdiag(\vec{e_i})\nonumber\\
 &=& B_{0,\vec{\nu}} + AI^{-1}A^{T}
\end{eqnarray}
where $A=\sqrt{V}diag(\vec{e_i})=A^{T}$ and $diag(\vec{e_i})$ denotes a diagonal matrix with diagonal vector $\vec{e_i}$.\newline

Therefore, using (\ref{LEM_ALT_REP_1_EQ1}) and applying the Sherman-Morrison-Woodbury identity as given in Lemma \ref{SWM_IDENTITY} in the Appendix, we obtain
\begin{eqnarray}\label{LEM_ALT_REP_1_EQ2}
 B_{1,\vec{\nu}}^{-1}
 &=& (B_{0,\vec{\nu}} + AI^{-1}A^{T})^{-1}\nonumber\\
 &=& B_{0,\vec{\nu}}^{-1} - B_{0,\vec{\nu}}^{-1}A(A^{T}B_{0,\vec{\nu}}^{-1}A+I)^{-1}A^{T}B_{0,\vec{\nu}}^{-1}.
\end{eqnarray}

Let $B_{0,\vec{\nu}}^{-1}=((b_{ij}))_{n \times n}$. Since the matrix $B_{0,\vec{\nu}}$ is positive definite, its inverse $B_{0,\vec{\nu}}^{-1}$ is also positive definite. Hence, $b_{ii}>0$ for all $i=1,\dots,n$.\newline

Next we observe that,
\begin{eqnarray}
 A^{T}B_{0,\vec{\nu}}^{-1}A
 &=& V \begin{pmatrix} 0 & 0 & \dots & 0 \\ \vdots \\ 0 & 0 & \dots & 0 \\ b_{1i} & b_{2i} & \dots & b_{ni} \\ 0 & 0 & \dots & 0\\ \vdots\\0 & 0 & \dots & 0\end{pmatrix}\begin{pmatrix}\vec{0}, & \dots, & \vec{0}, & \vec{e_{i}}, & \vec{0}, & \dots, & \vec{0}\end{pmatrix}\nonumber\\
 &=& Vdiag(0,\dots,0,b_{ii},0,\dots,0)\nonumber\\
 &=& V b_{ii} diag(\vec{e_i}).\label{LEM_ALT_REP_1_EQ3}
\end{eqnarray}

Therefore, using (\ref{LEM_ALT_REP_1_EQ3}) we get
\begin{equation}
 I+A^{T}B_{0,\vec{\nu}}^{-1}A=diag(1,\dots,1,1+Vb_{ii},1,\dots,1)\nonumber
\end{equation}
whence we have
\begin{equation}
 (I+A^{T}B_{0,\vec{\nu}}^{-1}A)^{-1}=diag(1,\dots,1,\frac{1}{1+Vb_{ii}},1,\dots,1).\label{LEM_ALT_REP_1_EQ4}
\end{equation}

Note that the matrix $A$ is symmetric, that is, $A=A^{T}$. Therefore, using (\ref{LEM_ALT_REP_1_EQ4}) and then applying exactly the same arguments used for proving (\ref{LEM_ALT_REP_1_EQ3}), we obtain
\begin{eqnarray}\label{LEM_ALT_REP_1_EQ5}
 A(I+A^{T}B_{0,\vec{\nu}}^{-1}A)^{-1}A^{T}
 &=& A^{T}(I+A^{T}B_{0,\vec{\nu}}^{-1}A)^{-1}A\nonumber\\
 &=& \frac{V}{1+Vb_{ii}}diag(\vec{e_i})\nonumber\\
 &=&\frac{1}{1+Vb_{ii}}AA^{T}.
\end{eqnarray}

Therefore, using (\ref{LEM_ALT_REP_1_EQ5}) we have
\begin{eqnarray}\label{LEM_ALT_REP_1_EQ6}
 B_{0,\vec{\nu}}^{-1}A(I+A^{T}B_{0,\vec{\nu}}^{-1}A)^{-1}A^{T}B_{0,\vec{\nu}}^{-1}&=& \frac{1}{1+Vb_{ii}}B_{0,\vec{\nu}}^{-1}AA^{T}B_{0,\vec{\nu}}^{-1}\nonumber\\
 &=& \frac{V}{1+Vb_{ii}}B_{0,\vec{\nu}}^{-1}diag(\vec{e_i})diag(\vec{e_i})^{T}B_{0,\vec{\nu}}^{-1}.
\end{eqnarray}

Now observe that 
\begin{equation}\label{LEM_ALT_REP_1_EQ7}
B_{0,\vec{\nu}}^{-1}diag(\vec{e_i})=
\begin{pmatrix}
\vec{0}, & \dots, & \vec{0}, & \vec{b_i}, & \vec{0}, & \dots, & \vec{0}
\end{pmatrix}. 
\end{equation}

Then, combining (\ref{LEM_ALT_REP_1_EQ6}) and (\ref{LEM_ALT_REP_1_EQ7}), it follows that
\begin{equation}
B_{0,\vec{\nu}}^{-1}A(I+A^{T}B_{0,\vec{\nu}}^{-1}A)^{-1}A^{T}B_{0,\vec{\nu}}^{-1} = \frac{V}{1+Vb_{ii}}\vec{b_i}\vec{b_i}^{T}.\label{LEM_ALT_REP_1_EQ8}
\end{equation}

On combining (\ref{LEM_ALT_REP_1_EQ2}) and (\ref{LEM_ALT_REP_1_EQ8}), the stated result then follows immediately.
\end{proof}

\begin{flushleft}
  \textbf{Proof of Lemma \ref{LEM_ALT_REP_2}}
\end{flushleft}

\begin{proof}
Let us write $B_{1,\vec{\nu}}=\vec{\Sigma}+VB_{\vec{\nu}:\nu_i=1}$ and $B_{0,\vec{\nu}}=\vec{\Sigma}+VB_{\vec{\nu}:\nu_i=0}$. Note that both $B_{1,\vec{\nu}}$ and $B_{0,\vec{\nu}}$ are positive definite, and so are their corresponding inverse matrices. Hence, $|B^{-1}_{0,\vec{\nu}}|=|B_{0,\vec{\nu}}|^{-1}>0$. Therefore,
\begin{equation}\label{LEM_ALT_REP_2_EQ1}
 \frac{|\vec{\Sigma}+VB_{\vec{\nu}:\nu_i=1}|}{|\vec{\Sigma}+VB_{\vec{\nu}:\nu_i=0}|}
 = \frac{|B_{1,\vec{\nu}} |}{| B_{0,\vec{\nu}}|}
 = |B^{-1}_{0,\vec{\nu}}B_{1,\vec{\nu}}|.
\end{equation}


Observe that, $B_{1,\vec{\nu}}=B_{0,\vec{\nu}}+Vdiag(\vec{e_i})$, where $\vec{e_i}$ denotes the unit vector with the $i$-th component unity and $diag(\vec{e_i})$ stands for a diagonal matrix with diagonal vector $\vec{e_i}$. Using this fact and applying the arguments employed in the proof of Lemma \ref{LEM_ALT_REP_1}, it follows that
 \begin{eqnarray}\label{LEM_ALT_REP_2_EQ2}
B^{-1}_{0,\vec{\nu}} B_{1,\vec{\nu}}
&=& I+VB^{-1}_{0,\vec{\nu}}diag(\vec{e_i})\nonumber\\
&=& I+V diag (\vec{0},\dots,\vec{0},\vec{b}_i(\vec{\nu}),\vec{0},\dots,\vec{0}).
\end{eqnarray}

Therefore, combining (\ref{LEM_ALT_REP_2_EQ1}) and (\ref{LEM_ALT_REP_2_EQ2}), we obtain
\begin{eqnarray}\label{LEM_ALT_REP_2_EQ3}
\frac{|\vec{\Sigma}+VB_{\vec{\nu}:\nu_i=1}|}{|\vec{\Sigma}+VB_{\vec{\nu}:\nu_i=0}|}
&=&|B^{-1}_{0,\vec{\nu}}B_{1,\vec{\nu}}|\nonumber\\
&=& |I+V diag (\vec{0},\dots,\vec{0},\vec{b}_i(\vec{\nu}),\vec{0},\dots,\vec{0})|\nonumber\\
&=& 1+Vb_{ii}(\vec{\nu}).
\end{eqnarray}

This completes the proof of Lemma \ref{LEM_ALT_REP_2}.
\end{proof}

\begin{flushleft}
 \textbf{Proof of Lemma \ref{LEM_ALT_REP_3}}
\end{flushleft}

\begin{proof}

First observe that there exists orthogonal matrices $P$ and $Q$ which respectively interchange the rows and column vectors of $\vec{\Sigma}+VB_{\vec{\nu}:\nu_i=0}$ in such a way that
\begin{equation}\label{LEM_ALT_REP_3_EQ6}
 P\big(\vec{\Sigma}+VB_{\vec{\nu}:\nu_i=0}\big)Q=\begin{pmatrix} \sigma_{ii} & \vec{\sigma}^{T}_{(-i)} \\ \vec{\sigma}_{(-i)} & (\vec{\Sigma}+VB_{\vec{\nu}:\nu_i=0})_{(-i,-i)}\end{pmatrix}
\end{equation}

Since $\sigma_{ii}-\vec{\sigma}^{T}_{(-i)}\big((\vec{\Sigma}+VB_{\vec{\nu}:\nu_i=0})_{(-i,-i)}\big)^{-1}\vec{\sigma}_{(-i)}>0$, using (\ref{LEM_ALT_REP_3_EQ6}) and Lemma \ref{INV_PART_MAT} in the Appendix, we obtain
\begin{eqnarray}\label{LEM_ALT_REP_3_EQ7}
 Q^{-1}(\vec{\Sigma}+VB_{\vec{\nu}:\nu_i=0})^{-1}P^{-1}
 =\begin{pmatrix}
d^{-1} & \vec{c}^{T}\\
\vec{c} & M      
\end{pmatrix},
\mbox{say.}
\end{eqnarray}

Here
\begin{equation}\label{LEM_ALT_REP_3_EQ8}
 d=\sigma_{ii}-\vec{\sigma}^{T}_{(-i)}\big((\vec{\Sigma}+VB_{\vec{\nu}:\nu_i=0})_{(-i,-i)}\big)^{-1}\vec{\sigma}_{(-i)}>0,
\end{equation}
\begin{eqnarray}\label{LEM_ALT_REP_3_EQ9}
 \vec{c}
 &=&-\big[\sigma_{ii}-\vec{\sigma}^{T}_{(-i)}\big((\vec{\Sigma}+VB_{\vec{\nu}:\nu_i=0})_{(-i,-i)}\big)^{-1}\vec{\sigma}_{(-i)}\big]^{-1}\big((\vec{\Sigma}+VB_{\vec{\nu}:\nu_i=0})_{(-i,-i)}\big)^{-1}\vec{\sigma}_{(-i)}\nonumber\\
 &=& - d\big((\vec{\Sigma}+VB_{\vec{\nu}:\nu_i=0})_{(-i,-i)}\big)^{-1}\vec{\sigma}_{(-i)} \mbox{ }[\mbox{using } (\ref{LEM_ALT_REP_3_EQ8})]
\end{eqnarray}
and
\begin{eqnarray}\label{LEM_ALT_REP_3_EQ10}
 M
 &=& \big((\vec{\Sigma}+VB_{\vec{\nu}:\nu_i=0})_{(-i,-i)}\big)^{-1}\nonumber\\
 &&\mbox{ }+\mbox{ }d^{-1}\big((\vec{\Sigma}+VB_{\vec{\nu}:\nu_i=0})_{(-i,-i)}\big)^{-1}\vec{\sigma}_{(-i)}\vec{\sigma}^{T}_{(-i)}\big((\vec{\Sigma}+VB_{\vec{\nu}:\nu_i=0})_{(-i,-i)}\big)^{-1}.\nonumber
\end{eqnarray}

It should now be carefully observed that $P$ and $Q$ are some appropriately chosen orthogonal matrices used to interchange the rows and column vectors respectively of the matrix $\vec{\Sigma}+VB_{\vec{\nu}:\nu_i=0}$. Therefore, by pre-multiplying and post-multiplying both sides of (\ref{LEM_ALT_REP_3_EQ7}) by $Q$ and $P$ respectively, we can get back the matrix $(\vec{\Sigma}+VB_{\vec{\nu}:\nu_i=0})^{-1}$. Hence, the constant $d^{-1}$ is noting but the term $b_{ii}(\vec{\nu})$, while the vector $\vec{c}$ as defined in (\ref{LEM_ALT_REP_3_EQ9}) above is simply the vector $(\vec{b}_{i}(\vec{\nu}))_{(-i)}$ itself. This completes the proof of Lemma \ref{LEM_ALT_REP_3}.
\end{proof}

\begin{flushleft}
 \textbf{Proof of Lemma \ref{LEM_ALT_REP_4}}
\end{flushleft}

\begin{proof}
Let us fix any $i\in\{1,\dots,n\}$ and take $\vec{\nu}_{(-i)}=\vec{0}$, that is, $\nu_j=0$ for all $j \neq i$. Also note that $\vec{\Sigma}+VB_{\vec{\nu}:\nu_i=0,\vec{\nu}_{(-i)}=\vec{0}}=\vec{\Sigma}$.\newline

Therefore, using Lemma \ref{LEM_ALT_REP_1}, for each fixed $\vec{x}\in\mathbb{R}^{n}$, we have
\begin{eqnarray}\label{LEM_ALT_REP_4_EQ1}
&&\vec{x}^{T}(\vec{\Sigma}+VB_{\vec{\nu}:\nu_i=0,\vec{\nu}_{(-i)}=\vec{0}})^{-1}\vec{x}-\vec{x}^{T}(\vec{\Sigma}+VB_{\vec{\nu}:\nu_i=1, \vec{\nu}_{(-i)}=\vec{0}})^{-1}\vec{x}\nonumber\\
&=& \vec{x}^{T}\vec{\Sigma}^{-1}\vec{x}-\bigg[\vec{x}^{T}\vec{\Sigma}^{-1}\vec{x} - 
    \frac{V}{1+Vb_{ii}}\bigg(\sum_{i=1}^{n}b_{ji}x_j\bigg)^2\bigg]\nonumber\\
&=& \frac{V}{1+Vb_{ii}}\bigg(\sum_{i=1}^{n}b_{ji}x_j\bigg)^2.
\end{eqnarray}

Again, using Lemma \ref{LEM_ALT_REP_2} we obtain
\begin{equation}\label{LEM_ALT_REP_4_EQ2}
 \frac{|\vec{\Sigma}+VB_{\vec{\nu}:\nu_i=0,\vec{\nu}_{(-i)}=\vec{0}}|}{|\vec{\Sigma}+VB_{\vec{\nu}:\nu_i=1,\vec{\nu}_{(-i)}=\vec{0}}|}
 =\frac{1}{1+Vb_{ii}}\cdot
\end{equation}

Therefore, using (\ref{LEM_ALT_REP_4_EQ1}) and (\ref{LEM_ALT_REP_4_EQ2}) we obtain
\begin{eqnarray}
 &&\frac{f(\vec{x}|\nu_i=1,\vec{\nu}_{(-i)}=\vec{0})}{f(\vec{x}|\nu_i=0,\vec{\nu}_{(-i)}=\vec{0})}\nonumber\\
 &=& \frac{(2\pi)^{-\frac{n}{2}}|\vec{\Sigma}+VB_{\vec{\nu}:\nu_i=1,\vec{\nu}_{(-i)}=\vec{0}}|^{-1/2}\exp\big\{-\frac{1}{2}\vec{x}^{T}(\vec{\Sigma}+VB_{\vec{\nu}:\nu_i=1,\vec{\nu}_{(-i)}=\vec{0}})^{-1}\vec{x}\big\}}{(2\pi)^{-\frac{n}{2}}|\vec{\Sigma}+VB_{\vec{\nu}:\nu_i=1,\vec{\nu}_{(-i)}=\vec{0}}|^{-1/2}\exp\big\{-\frac{1}{2}\vec{x}^{T}(\vec{\Sigma}+VB_{\vec{\nu}:\nu_i=1,\vec{\nu}_{(-i)}=\vec{0}})^{-1}\vec{x}\big\}}\nonumber\\
 &=& \frac{1}{\sqrt{1+Vb_{ii}}}\exp\big\{\frac{1}{2}\vec{x}^{T}(\vec{\Sigma}+VB_{\vec{\nu}:\nu_i=0,\vec{\nu}_{(-i)}=\vec{0}})^{-1}\vec{x}-\frac{1}{2}\vec{x}^{T}(\vec{\Sigma}+VB_{\vec{\nu}:\nu_i=1, \vec{\nu}_{(-i)}=\vec{0}})^{-1}\vec{x} \big\}\nonumber\\
 &=& \frac{1}{\sqrt{1+Vb_{ii}}}\times \exp\bigg\{\frac{V}{2(1+Vb_{ii})} \bigg(\sum_{i=1}^{n}b_{ji}x_j\bigg)^2\bigg\}.\nonumber 
\end{eqnarray}

This completes the proof of Lemma \ref{LEM_ALT_REP_4}.
\end{proof}

\begin{flushleft}
\textbf{Proof of Theorem \ref{THM_ALTERNATIVE_REPRESENTATION}}
\end{flushleft}

\begin{proof}
Proof of this theorem follows by employing exactly the same line of arguments used for proving Lemma \ref{LEM_ALT_REP_1} - Lemma \ref{LEM_ALT_REP_4} by taking into consideration the data vector $\vec{X}^{(i_1,\dots,i_{t-1})}$ and its corresponding variance-covariance matrix $\vec{\Sigma}_{(i_1,\dots,i_{t-1})}$.
\end{proof}

\bibliographystyle{apalike}
\bibliography{prasenjit_thesis_reference}

\begin{thebibliography}{}

\bibitem[Benjamini and Heller, 2007]{BH2007}
Benjamini, Y. and Heller, R. (2007).
\newblock False discovery rate for spatial signals.
\newblock {\em Journal of American Statistical Association}, 102:1272--1281.

\bibitem[Benjamini and Hochberg, 1995]{BH1995}
Benjamini, Y. and Hochberg, Y. (1995).
\newblock Controlling the false discovery rate: a practical and powerfull
  approach to multiple testing.
\newblock {\em Journal of the Royal Statistical Society, Series B},
  57(1):289--300.

\bibitem[Benjamini and Hochberg, 2000]{BH2000}
Benjamini, Y. and Hochberg, Y. (2000).
\newblock On the adaptive control of the false discovery rate in multiple
  testing with independent statistics.
\newblock {\em Journal of Educational and Behavioral Statistics}, 25:60--83.

\bibitem[Benjamini et~al., 2006]{BKY2006}
Benjamini, Y., Krieger, A.~M., and Yekutieli, D. (2006).
\newblock Adaptive linear step-up procedures that control the false discovery
  rate.
\newblock {\em Biometrika}, 93(3):491--507.

\bibitem[Benjamini and Liu, 1999]{BL1999}
Benjamini, Y. and Liu, W. (1999).
\newblock A step-down multiple hypotheses that controls the false discover rate
  under independence.
\newblock {\em Journal of Statistical Planning and Inference}, 82:163--170.

\bibitem[Benjamini and Yekutieli, 2001]{BY2001}
Benjamini, Y. and Yekutieli, D. (2001).
\newblock The control of the false discovery rate in multiple testing under
  dependency.
\newblock {\em The Annals of Statistics}, 29(4):1165--1188.

\bibitem[Blanchard and Roquain, 2009]{BLROQ2009}
Blanchard, G. and Roquain, E. (2009).
\newblock Adaptive fdr control under independence and dependence.
\newblock {\em Journal of Machine Learning Research}, 10:2837--2871.

\bibitem[Bogdan et~al., 2011]{BCFG2011}
Bogdan, M., Chakrabarti, A., Frommlet, F., and Ghosh, J.~K. (2011).
\newblock Asymptotic bayes-optimality under sparsity of some multiple testing
  procedures.
\newblock {\em The Annals of Statistics}, 39(3):1551--1579.

\bibitem[Cai and Jin, 2010]{CAI_JIN_2010}
Cai, T.~T. and Jin, J. (2010).
\newblock Optimal rates of convergence for estimating the null density and
  proportion of nonnull effects in large-scale multiple comparisons.
\newblock {\em The Annals of Statistics}, 38(1):100--145.

\bibitem[Chen and Sarkar, 2004]{CHEN_SARKAR_2004}
Chen, J. and Sarkar, S.~K. (2004).
\newblock Multiple testing of response rates with a control: a {B}ayesian
  stepwise approach.
\newblock {\em Journal of Statistical Planning and Inference}, 125:3--16.

\bibitem[Chi, 2008]{CHI2008}
Chi, Z. (2008).
\newblock False discovery rate control with multivariate p-values.
\newblock {\em Electronic Journal of Statistics}, 2:368--411.

\bibitem[Cohen et~al., 2007]{COHEN_KOL_SACK_2007}
Cohen, A., Kolassa, J., and Sackrowitz, H.~B. (2007).
\newblock A smooth version of the step-up procedure for multiple tests of
  hypotheses.
\newblock {\em Journal of Statistical Planning and Inference},
  137(11):3352--3360.

\bibitem[Cohen and Sackrowitz, 2005a]{COHEN_SACK_2005_B}
Cohen, A. and Sackrowitz, H.~B. (2005a).
\newblock Characterization of {B}ayes procedures for multiple endpoint problems
  and inadmissibility of the step-up procedure.
\newblock {\em The Annals of Statistics}, 33(1):145--158.

\bibitem[Cohen and Sackrowitz, 2005b]{COHEN_SACK_2005_A}
Cohen, A. and Sackrowitz, H.~B. (2005b).
\newblock Decision theory results for one-sided multiple comparison procedures.
\newblock {\em The Annals of Statistics}, 33(1):126--144.

\bibitem[Cohen and Sackrowitz, 2007]{COHEN_SACK_2007}
Cohen, A. and Sackrowitz, H.~B. (2007).
\newblock More on the inadmissibility of step-up.
\newblock {\em Journal of Multivariate Analysis}, 98(3):481--492.

\bibitem[Cohen and Sackrowitz, 2008]{COHEN_SACK_2008}
Cohen, A. and Sackrowitz, H.~B. (2008).
\newblock Multiple testing of two-sided alternatives with dependent data.
\newblock {\em Statistica Sinica}, 18(4):1593--1602.

\bibitem[Cohen et~al., 2009]{COHEN_SACK_XU_2009}
Cohen, A., Sackrowitz, H.~B., and Xu, M. (2009).
\newblock A new multiple testing method in the dependent case.
\newblock {\em The Annals of Statistics}, 37(3):1518--1544.

\bibitem[Dudoit et~al., 2003]{DSB2003}
Dudoit, S., Shaffer, J.~P., and Boldrick, J.~C. (2003).
\newblock Multiple hypothesis testing in microarray experiments.
\newblock {\em Statistical Science}, 18(1):71--103.

\bibitem[Efron, 2004]{Efron2004}
Efron, B. (2004).
\newblock Large-scale simultaneous hypothesis testing: The choice of a null
  hypothesis.
\newblock {\em Journal of American Statistical Association}, 99(465):96--104.

\bibitem[Efron, 2007]{Efron_2007}
Efron, B. (2007).
\newblock Correlation and large-scale simultaneous significance testing.
\newblock {\em Journal of American Statistical Association}, 102(477):93--103.

\bibitem[Efron et~al., 2001]{ETST_2001}
Efron, B., Tibshirani, R., Storey, J.~D., and Tusher, V. (2001).
\newblock Empirical bayes analysis of a microarray experiment.
\newblock {\em Journal of American Statistical Association},
  96(456):1151--1160.

\bibitem[Fan et~al., 2012]{FHG_2012}
Fan, J., Han, X., and Gu, W. (2012).
\newblock Estimating false discovery proportion under arbitrary covariance
  dependence.
\newblock {\em Journal of American Statistical Association},
  107(499):1019--1035.

\bibitem[Finner et~al., 2009]{FDR2009}
Finner, H., Dickhaus, T., and Roters, M. (2009).
\newblock On the false discovery rate and an asymptotically optimal rejection
  curve.
\newblock {\em The Annals of Statistics}, 37:596--618.

\bibitem[Finner and Roters, 2001]{FH2001}
Finner, H. and Roters, M. (2001).
\newblock On the false discovery rate and expected type i errors.
\newblock {\em Biometrical Journal}, 43:985--1005.

\bibitem[Finner and Roters, 2002]{FH2002}
Finner, H. and Roters, M. (2002).
\newblock Multiple hypotheses testing and expected number of type 1 errors.
\newblock {\em The Annals of Statistics}, 30:220--238.

\bibitem[Finner and Strassburger, 2002]{FS2002}
Finner, H. and Strassburger, K. (2002).
\newblock The partitioning principle: A powerful tool in multiple decision
  theory.
\newblock {\em The Annals of Statistics}, 30:1194--1213.

\bibitem[Friguet et~al., 2009]{FKC_2009}
Friguet, C., Kloareg, M., and Causeur, D. (2009).
\newblock A factor model approach to multiple testing under dependence.
\newblock {\em Journal of American Statistical Association},
  104(488):1406--1415.

\bibitem[Gavrilov et~al., 2009]{GBS2009}
Gavrilov, Y., Benjamini, Y., and Sarkar, S.~K. (2009).
\newblock An adaptive step-down procedure with proven fdr control under
  independence.
\newblock {\em The Annals of Statistics}, 37(2):619--629.

\bibitem[Genovese et~al., 2006]{GRW2006}
Genovese, C., Roeder, K., and Wasserman, L. (2006).
\newblock False discovery control with $p$-value weighting.
\newblock {\em Biometrika}, 93:509--524.

\bibitem[Genovese and Wasserman, 2002]{GW2002}
Genovese, C. and Wasserman, L. (2002).
\newblock Operating characteristics and extensions of the false discovery rate
  procedure.
\newblock {\em Journal of the Royal Statistical Society Series B}, 64:499--517.

\bibitem[Genovese and Wasserman, 2004]{GW_2004}
Genovese, C. and Wasserman, L. (2004).
\newblock A stochastic process approach to false discovery control.
\newblock {\em The Annals of Statistics}, 32:1035--1061.

\bibitem[Gordon et~al., 2007]{GGQY2007}
Gordon, A., Glazko, G., Qiu, X., and Yakovlev, A. (2007).
\newblock Control of the mean number of false discoveries, bonferroni, and
  stability of multiple testing.
\newblock {\em The Annals of Applied Statistics}, 1:179--190.

\bibitem[Guo, 2009]{GUO_2009}
Guo, W. (2009).
\newblock A note on adaptive bonferroni and holm procedures under dependence.
\newblock {\em Biometrika}, 96(4):1012--1018.

\bibitem[Guo et~al., 2014]{GHS2014}
Guo, W., He, L., and Sarkar, S.~K. (2014).
\newblock Further results on controlling the false discovery proportion.
\newblock {\em The Annals of Statistics}, 42(3):1070--1101.

\bibitem[Guo and Rao, 2008]{GR2008}
Guo, W. and Rao, M.~B. (2008).
\newblock On optimality of the benjamini-hochberg procedure for the false
  discovery rate.
\newblock {\em Statistics and Probability Letters}, 78:2024--2030.

\bibitem[Hall and Jin, 2010]{HJ2010}
Hall, P. and Jin, J. (2010).
\newblock Innovated higher criticism for detecting sparse signals in correlated
  noise.
\newblock {\em The Annals of Statistics}, 38(3):1686--1732.

\bibitem[Jin, 2006]{JIN_2006}
Jin, J. (2006).
\newblock Proportion of nonzero normal means: Universal oracle equivalences and
  uniformly consistent estimations.
\newblock {\em Technical report, Purdue University, Department of Statistics}.

\bibitem[Jin, 2008]{JIN_2008}
Jin, J. (2008).
\newblock Proportion of non-zero normal means: universal oracle equivalences
  and uniformly consistent estimators.
\newblock {\em Journal of Royal Statistical Society Series B}, 70(3):461--493.

\bibitem[Jin and Cai, 2007]{JIN_CAI_2007}
Jin, J. and Cai, T.~T. (2007).
\newblock Estimating the null and the proportion of nonnull effects in
  large-scale multiple comparisons.
\newblock {\em Journal of American Statistical Association}, 102(478):495--506.

\bibitem[Klebanov and Yakovlev, 2007]{KY2007}
Klebanov, L. and Yakovlev, A. (2007).
\newblock Diverse correlation structures in gene expression data and their
  utility in improving statistical inference.
\newblock {\em The Annals of Applied Statistics}, 1(2):538--559.

\bibitem[Leek and Storey, 2008]{LEEK_STO_2008}
Leek, J. and Storey, J.~D. (2008).
\newblock A general framework for multiple testing dependence.
\newblock {\em Proceedings of the National Academy of Sciences},
  105(48):18718--18723.

\bibitem[Lehmann, 1957a]{LEHMANN1957a}
Lehmann, E.~L. (1957a).
\newblock A theory of some multiple decision problems (i).
\newblock {\em The Annals of Mathematical Statistics}, 28:1--25.

\bibitem[Lehmann, 1957b]{LEHMANN1957b}
Lehmann, E.~L. (1957b).
\newblock A theory of some multiple decision problems (ii).
\newblock {\em The Annals of Mathematical Statistics}, 28:547--572.

\bibitem[Lehmann and Romano, 2005]{LR2005}
Lehmann, E.~L. and Romano, J.~P. (2005).
\newblock Generalizations of the familywise error rate.
\newblock {\em The Annals of Statistics}, 33(3):1138--1154.

\bibitem[Lehmann et~al., 2005]{LRS2005}
Lehmann, E.~L., Romano, J.~P., and Shaffer, J.~P. (2005).
\newblock On optimality of stepdown and stepup multiple test procedures.
\newblock {\em The Annals of Statistics}, 33:1084--1108.

\bibitem[Matthes and Truax, 1967]{MAT_TRU_1967}
Matthes, T.~K. and Truax, D.~R. (1967).
\newblock Tests of composite hypotheses for the multivariate exponential
  family.
\newblock {\em The Annals of Mathematical Statistics}, 38:681--697.

\bibitem[Meinshausen and Rice, 2006]{MR2006}
Meinshausen, N. and Rice, J. (2006).
\newblock Estimating the proportion of false null hypotheses among a large
  number of independently tested hypotheses.
\newblock {\em The Annals of Statistics}, 34(1):373--393.

\bibitem[Neuvial and Roquain, 2012]{NR2012}
Neuvial, P. and Roquain, E. (2012).
\newblock On false discovery rate thresholding for classification under
  sparsity.
\newblock {\em The Annals of Statistics}, 40(5):2572--2600.

\bibitem[Owen, 2005]{OWEN2005}
Owen, A.~B. (2005).
\newblock Variance of the number of false discoveries.
\newblock {\em Journal of the Royal Statistical Society Series B}, 67:411--426.

\bibitem[Pollard and van~der Laan, 2002]{POLVDL2002}
Pollard, K.~S. and van~der Laan, M.~J. (2002).
\newblock Resampling-based multiple testing: Asymptotic control of type i error
  and applications to gene expression data.
\newblock {\em Journal of Statistical Planning and Research}, 125:85--100.

\bibitem[Qiu et~al., 2005a]{QBKY2005}
Qiu, X., Brooks, A., Klebanov, L., and Yakovlev, A.~Y. (2005a).
\newblock The effects of normalization on the correlation structure of
  microarray data.
\newblock {\em BMC Bioinformatics}, 6.

\bibitem[Qiu et~al., 2005b]{QKY2005}
Qiu, X., Klebanov, L., and Yakovlev, A.~Y. (2005b).
\newblock Correlation between gene expression levels and limitations of the
  empirical bayes methodology for finding differentially expressed genes.
\newblock {\em Statistical Applications in Genetic and Molecular Biology}, 4.

\bibitem[Qiu et~al., 2007]{QXGY2007}
Qiu, X., Xiao, Y., Gordon, A., and Yakovlev, A. (2007).
\newblock Assessing stability of gene selection in microarray data analysis.
\newblock {\em BMC Bioinformatics}, 7.

\bibitem[Romano and Shaikh, 2006]{RS2006}
Romano, J.~P. and Shaikh, A.~M. (2006).
\newblock Stepup procedures for control of generalizations of the familywise
  error rate.
\newblock {\em The Annals of Statistics}, 34(4):1850--1873.

\bibitem[Romano et~al., 2008]{ROM_SHK_WOLF_2008}
Romano, J.~P., Shaikh, A.~M., and Wolf, M. (2008).
\newblock Control of the false discovery rate under dependence using the
  bootstrap and subsampling.
\newblock {\em TEST}, 17(3):417--442.

\bibitem[Santner et~al., 2003]{SANT_WILL_NOTZ}
Santner, T.~J., Williams, B.~J., and Notz, W.~I. (2003).
\newblock {\em The Design and Analysis of Computer Experiments}.

\bibitem[Sarkar, 2002]{SKS_2002}
Sarkar, S.~K. (2002).
\newblock Some results on false discovery rate in stepwise multiple testing
  procedures.
\newblock {\em The Annals of Statistics}, 30(1):239--257.

\bibitem[Sarkar, 2007]{SKS_2007}
Sarkar, S.~K. (2007).
\newblock Stepup procedures controlling generalized fwer and generalized fdr.
\newblock {\em The Annals of Statistics}, 35(6):2405--2420.

\bibitem[Sarkar, 2008]{SKS2008}
Sarkar, S.~K. (2008).
\newblock On methods controlling the false discovery rate.
\newblock {\em Sankhy{\={a}}: The Indian Journal of Statistics},
  70-A(2):135--168.

\bibitem[Sarkar and Guo, 2009]{SG2009}
Sarkar, S.~K. and Guo, W. (2009).
\newblock On a generalized false discovery rate.
\newblock {\em The Annals of Statistics}, 37(3):1545--1565.

\bibitem[Storey, 2002]{STO_2002}
Storey, J.~D. (2002).
\newblock A direct approach to false discovery rates.
\newblock {\em Journal of Royal Statistical Society Series B}, 64(3):479--498.

\bibitem[Storey et~al., 2004]{STS_2004}
Storey, J.~D., Taylor, J.~E., and Siegmund, D. (2004).
\newblock Strong control, conservative point estimation and simultaneous
  conservative consistency of false discovery rates: A unified approach.
\newblock {\em Journal of Royal Statistical Society Series B}, 66(1):187--205.

\bibitem[Sun and Cai, 2009]{SUN_CAI_2009}
Sun, W. and Cai, T.~T. (2009).
\newblock Large-scale multiple testing under dependence.
\newblock {\em Journal of the Royal Statistical Society. Series B},
  71(2):393--424.

\bibitem[Xie et~al., 2011]{XCML_2011}
Xie, J., Cai, T.~T., Maris, J., and Li, H. (2011).
\newblock Optimal false discovery rate control for dependent data.
\newblock {\em Statistics and its Interface}, 4:417--430.

\bibitem[Yekutieli and Benjamini, 1999]{YB1999}
Yekutieli, D. and Benjamini, Y. (1999).
\newblock Resampling-based false discovery rate controlling multiple test
  procedures for correlated test statistics.
\newblock {\em Journal of Statistical Planning and Research}, 82:171--196.

\end{thebibliography}


\end{document}